\documentclass[11pt,reqno]{amsart}

\usepackage{enumerate, amsmath, amsthm, amsfonts, amssymb, 
mathrsfs, graphicx, paralist, ulem}
\usepackage[usenames, dvipsnames]{color}
\usepackage[margin=1in]{geometry} 
\usepackage{hyperref}
\usepackage{comment}
\usepackage{bbm}

\usepackage{mathabx} 
\usepackage{breqn} 
\usepackage{tikz-cd}
\usepackage{amsrefs}

\numberwithin{equation}{section}
\newtheorem{Theorem}[equation]{Theorem}
\newtheorem{Proposition}[equation]{Proposition}
\newtheorem{Lemma}[equation]{Lemma}
\newtheorem{Corollary}[equation]{Corollary}

\theoremstyle{definition}
\newtheorem{Remark}[equation]{Remark}

\newtheorem{eg}[equation]{Example}
\newenvironment{example}[1][]{\begin{eg}[#1] \pushQED{\qed}}{\popQED \end{eg}}
\newtheorem{Definition}[equation]{Definition}

\newtheoremstyle{named}{}{}{\itshape}{}{\bfseries}{.}{.5em}{#1\thmnote{ #3}}
\theoremstyle{named}
\newtheorem*{namedtheorem}{Theorem}

\newcommand{\cA}{\mathcal{A}}

\newcommand{\bB}{\mathbf{B}}
\newcommand{\cB}{\mathcal{B}}

\newcommand{\cC}{\mathcal{C}}

\newcommand{\sE}{\mathsf{E}}

\newcommand{\sF}{\mathsf{F}}

\newcommand{\sG}{\mathsf{G}}

\newcommand{\sH}{\mathsf{H}}

\newcommand{\cL}{\mathcal{L}}

\newcommand{\bM}{\mathbf{M}}
\newcommand{\cM}{\mathcal{M}}

\newcommand{\sM}{\mathsf{M}}

\newcommand{\cO}{\mathcal{O}}

\newcommand{\cS}{\mathcal{S}}

\newcommand{\cT}{\mathcal{T}}

\newcommand{\cV}{\mathcal{V}}

\newcommand{\bX}{\mathbf{X}}

\newcommand{\ba}{\mathbf{a}}
\newcommand{\fa}{\mathfrak{a}}

\newcommand{\bb}{\mathbf{b}}
\newcommand{\fb}{\mathfrak{b}}

\newcommand{\fg}{\mathfrak{g}}

\newcommand{\fh}{\mathfrak{h}}

\newcommand{\kk}{\mathbbm{k}}

\newcommand{\st}{\mathsf{t}}

\newcommand{\CC}{\mathbb{C}}

\newcommand{\NN}{\mathbb{N}}

\newcommand{\PP}{\mathbb{P}}

\newcommand{\RR}{\mathbb{R}}

\newcommand{\ZZ}{\mathbb{Z}}

\renewcommand{\phi}{\varphi}
\renewcommand{\emptyset}{\varnothing}
\newcommand{\eps}{\varepsilon}

\renewcommand{\tilde}[1]{\widetilde{#1}}

\newcommand{\ul}[1]{\underline{#1}}

\newcommand{\sh}[1]{[#1]}

\DeclareMathOperator{\End}{End}

\DeclareMathOperator{\SYT}{SYT}

\newcommand{\SL}{SL}

\newcommand{\fsl}{\mathfrak{sl}}

\newcommand{\bone}{\mathbbm{1}}
\newcommand{\sone}{\mathsf{1}}
\newcommand{\mmod}{\mathrm{-mod}}
\newcommand{\proj}{\mathrm{-proj}}
\newcommand{\soc}{\mathrm{soc}}

\newcommand{\wt}{\text{wt}}

\newcommand\iso\cong
\newcommand\into\hookrightarrow
\newcommand\onto\twoheadrightarrow
\newcommand{\Hom}{\operatorname{Hom}}
\newcommand{\nc}{\newcommand}
\nc{\la}{\lambda}
\nc{\Iso}{\mathsf{Iso}}
\nc{\Irr}{\mathsf{Irr}}
\nc{\Id}{\mathrm{Id}}

\begin{document}

\title[Categorical braid group actions and cactus groups]{Categorical braid group actions \\ and cactus groups}
\author{Iva Halacheva}
\address{I.~Halacheva: Department of Mathematics, Northeastern University, USA}
\email{i.halacheva@northeastern.edu}

\author{Anthony Licata}
\address{A.~Licata: Mathematical Sciences Institute, Australian National University, Australia}
\email{anthony.licata@anu.edu.au}

\author{Ivan Losev}
\address{I.~Losev: Department of Mathematics, Yale University, and School of Mathematics, IAS, USA}
\email{ivan.loseu@gmail.com}

\author{Oded Yacobi}
\address{O.~Yacobi: School of Mathematics and Statistics, University of Sydney, Australia}
\email{oded.yacobi@sydney.edu.au}
\date{}
\maketitle

\begin{abstract}
Let $\fg$ be a semisimple simply-laced Lie algebra of finite type.  Let $\cC$ be an abelian categorical representation of the quantum group $U_q(\fg)$ categorifying an integrable representation $V$.  
The Artin braid group $B$ of $\fg$ acts on $D^b(\cC)$ by Rickard complexes, providing a triangulated equivalence 
\begin{align*}
\Theta_{w_0}:D^b(\cC_\mu) \to D^b(\cC_{w_0(\mu)})
\end{align*}
where $\mu$ is a weight of $V$, and $\Theta_{w_0}$ is a positive lift of the longest element of the Weyl group.  

We prove that this equivalence is t-exact up to shift when $V$ is isotypic, generalising a fundamental result of Chuang and Rouquier in the case $\fg=\fsl_2$.
For general $V$, we prove that $\Theta_{w_0}$ is a perverse equivalence with respect to a Jordan-H\"older filtration of $\cC$.  

Using these results we construct, from the action of 
$B$ on $V$, an action of the cactus group on the crystal of $V$.  This recovers the  cactus group action on $V$ defined via generalised Sch\"utzenberger involutions, and  provides a new connection between categorical representation theory and crystal bases.
We also use these results to give new  proofs of theorems of Berenstein-Zelevinsky, Rhoades, and Stembridge regarding the action of symmetric group on the Kazhdan-Lusztig basis of its Specht modules.  
\end{abstract}

\section{Introduction}
In their seminal work, Chuang and Rouquier introduced $\fsl_2$ categorifications on abelian categories \cite{CR}.  Their definition mirrors the notion of an $\fsl_2$ representation on a vector space: weight spaces are replaced by weight categories, Chevalley generators acting on them are replaced by Chevalley functors, and Lie algebra relations are replaced by isomorphisms of functors.  But, crucially, these isomorphisms are part of the ``higher data'' of categorification.

The richness of this theory was immediately evident.  As a corollary of an $\fsl_2$ categorification on representations of symmetric groups in positive characteristic, Chuang and Rouquier proved Broue's abelian defect conjecture in that case.  The essential tool allowing them to do this is the Rickard complex, which is a categorical lifting of the reflection matrix in $\SL_2$, and provides a derived equivalence between opposite weight categories.

Subsequently, Rouquier and Khovanov-Lauda vastly generalised this theory to  quantum symmetrisable Kac-Moody algebras $U_q(\fg)$ \cite{Rou2KM,KLI,KLII}.  Let $\kk$ be any field.  A graded abelian $\kk$-linear category $\cC$ endowed with a categorical representation of $U_q(\fg)$ possesses a family of Rickard complexes $\Theta_i$, indexed by the simple roots of $\fg$, acting on the derived category $D^b(\cC)$.  

Henceforth let $\fg$ be a semisimple simply-laced Lie algebra of finite type with Dynkin diagram $I$, $W$  its Weyl group, and $B$ its Artin braid group. Let $\cC$ be a categorical representation of $U_q(\fg)$ as in the previous paragraph.  Cautis and Kamnitzer proved that Rickard complexes satisfy the braid relations, as conjectured by Rouquier \cite{CK3}.  This defines an action of $B$ on $D^b(\cC)$, and is our main object of study.    

Categorical braid group actions defined via Rickard complexes have many significant applications.  For example, in low dimensional topology, the type A link homology theories (in particular Khovanov homology) emerge as a byproduct of these types of categorical braid group actions \cite{CK, CKII, LQR}.  In mirror symmetry, the theory of spherical twists plays an important role, and these all arise from categorical $\fsl_2$ representations \cite{ST01}.  

To describe our first theorem, recall that minimal categorifications are certain distinguished categorifications of simple representations.  On these the Rickard complex $\Theta_i$  is t-exact up to shift \cite[Theorem 6.6]{CR}.  Notice that this is a result about $\fsl_2$ categorifications, and in fact, this is one of Chuang-Rouquier's key technical results which they use to prove the derived equivalence.  

We generalise this result to $U_q(\fg)$, where we show that the composition of Rickard complexes corresponding to a positive lift of the longest element $w_0 \in W$ is t-exact up to shift on any isotypic categorification.  More precisely:

\begin{namedtheorem}[A][Theorem \ref{thm:mincat} \& Corollary \ref{cor:mainthm}]\label{thm:A}
Let $\cC$ be a categorical representation of $U_q(\fg)$ categorifying an isotypic representation of type $\lambda$, where $\lambda$ is a dominant integral weight.  Let $\mu$ be any weight, and let $n$ be the height of $\mu-w_0(\lambda)$.  Then the derived equivalence
\begin{align*}
\Theta_{w_0}\bone_\mu[n]:D^b(\cC_\mu) \to D^b(\cC_{w_0(\mu)})
\end{align*}
is t-exact.  
\end{namedtheorem}

This theorem is the technical heart of the paper.
In order to prove it we introduce a new combinatorial notion of ``marked words'' (Section \ref{sec:commrels}).  This allows us to use relations between $\Theta_i$ and Chevalley functors established by Cautis and Kamnitzer to deduce  the commutation relations involving $\Theta_{w_0}$ (Proposition \ref{prop:mainrel}).  We then use these relations to prove the theorem by induction on $n$.     

Our second theorem describes $\Theta_{w_0}$ on an arbitrary categorical representation of $U_q(\fg)$, also generalising a result of Chuang-Rouquier in the case $\fg=\fsl_2$.  Indeed, their study of the Rickard complex on an $\fsl_2$ categorification led them to define the notion of a ``perverse equivalence'' \cite{CRperv}.  

Consider an equivalence of triangulated categories $\sF:\cT\to\cT'$ with t-structures \cite{BBD}.  Suppose further that $\cT$ (respectively $\cT'$) is filtered by thick triangulated subcategories 
\begin{align*}
0\subset \cT_0 \subset \cdots \subset \cT_r=\cT, \quad 0\subset \cT_0' \subset \cdots \subset \cT_r'=\cT',
\end{align*}
and $\sF$ is compatible with these filtrations (cf. Section \ref{sec:pervdef} for precise definitions).  Then, roughly speaking,  $\sF$ is a perverse equivalence if on each subquotient $\sF:\cT_i/\cT_{i-1} \to \cT_i'/\cT_{i-1}'$ is t-exact up to shift.  

Since their introduction, perverse equivalences have proven useful in various contexts (e.g. representations of finite groups \cite{Cra-Rou}, geometric representation theory and mirror symmetry \cite{Agan}, and algebraic combinatorics \cite{GJN}).  Our second theorem shows that perverse equivalences are ubiquitous in categorical representation theory:

\begin{namedtheorem}[B][Theorem \ref{thm:perveq}]\label{thm:B}
Let $\cC$ be a categorical representation of $U_q(\fg)$, and let $\mu$ be any weight.  The derived equivalence
$
\Theta_{w_0}\bone_\mu:D^b(\cC_\mu) \to D^b(\cC_{w_0(\mu)})
$
is a perverse equivalence with respect to a Jordan-H\"older or isotypic filtration of $\cC$.  
\end{namedtheorem}

Note that if $J\subseteq I$ is a subdiagram, and $w_0^J$ is corresponding longest element, then this theorem implies that $\Theta_{w_0^J}\bone_\mu$ is a perverse equivalence for any $J$.  We also remark that our argument go through  in the ungraded setting, where $\cC$ is a categorical representation of $\fg$.  

Let us explain the filtration arising in Theorem B more precisely.  We apply Rouquier's Jordan-H\"older theory for representations of $2$-Kac-Moody algebras to our setting \cite{Rou2KM}.  We thus obtain a filtration of $\cC$ by Serre subcategories, 
\begin{align*}
0\subset \cC_0 \subset \cdots \subset \cC_r=\cC,
\end{align*}
such that each factor $\cC_i$ is a subrepresentation, and each subquotient $\cC_i/\cC_{i-1}$ categorifies either a simple module (Theorem \ref{thm:JH}) or an isotypic component (Remark \ref{rem:isofilts}).  Then   $
\Theta_{w_0}\bone_\mu$ is a perverse equivalence with respect to the filtration whose $i$-th filtered component consists of complexes in $D^b(\cC_\mu)$ with cohomology supported in $\cC_i$.

We remark that in the case $\fg=\fsl_2$ this gives a more conceptual proof of a result of Chuang-Rouquier \cite[Proposition 8.4]{CRperv}.  If $\cC$ is the tensor product categorification of the $n$-fold tensor product of the standard representation of $\fsl_n$, we recover a theorem of the third author \cite{LosCacti}.  We explain this in Example \ref{ex:catO}, where we show how to interpret the filtration on the principal block of the BGG category $\cO$ using the Robinson-Schensted correspondence. 

In fact, the third author and Bezrukavnikov formulated a principle that suitable categorical braid group representations should have a ``crystal limit'' \cite[Section 9]{BLet}.  As an application of our results we can make this precise in the setting of categorical representations of $U_q(\fg)$.

Recall that to an integrable representation $V$ of $U_q(\fg)$, Kashiwara associated its crystal basis $\bB_V$ \cite{Kash90}, which is closely related to Lusztig's canonical basis \cite{GL92}.  If $V$ is categorified by $\cC$ then there is a natural identification $\bB_V=\Irr(\cC)$, the set of isomorphism classes of simple objects in $\cC$ up to shift (cf. Proposition \ref{prop:cryst})	.

One of the most important features of the theory is the existence of a tensor product, endowing the category of crystals with a monoidal structure. The commutator of crystals is controlled by a group called the cactus group, just as $B$ controls the commutator in the category of representations of $U_q(\fg)$ \cite{HK06}.  There is also an internal cactus group action, mirroring Lusztig's internal braid group action on $V$.  Indeed, there is a cactus group $C$ associated to $\fg$ (or rather to its Dynkin diagram $I$), which can be presented by generators $c_J$ indexed by connected subdiagrams $J \subseteq I$ (cf. Section \ref{sec:cactus}).  Then $C$ acts on $\bB_V$ via the so-called Sch\"utzenberger involutions (cf. Theorem \ref{thm:cactact}).  

So, starting with an integrable representation $V$ of the quantum group we obtain:
an action of $B$ on V,
a $\fg$-crystal $\bB_V$, and
an action of $C$ on $\bB_V$.
We schematically picture this situation as follows:
\[
\begin{tikzcd}
U_q(\fg) \curvearrowright V  \ar[rr]  \ar[d] &  & B \curvearrowright V  \arrow[d, dashed,"?"] \\
\fg\text{-crystal } \bB_V \ar[rr]            &  & C       \curvearrowright    \bB_V        
\end{tikzcd}
\]
Naturally one asks: can we ``crystallise'' the braid group action on $V$ directly to obtain the cactus group action on $\bB_V$?  Our results allow us to answer this in the affirmative.    

The key point is that a perverse equivalence $\sF:\cT \to \cT'$ induces a bijection $\Irr(\cT^\heartsuit) \leftrightarrow \Irr((\cT')^\heartsuit)$, where $\cT^\heartsuit$ denotes the heart of the t-structure. In the setting of Theorem B, we obtain a bijection $\varphi_I:\Irr(\cC) \to \Irr(\cC)$.  In fact, if $J\subseteq I$ is a subdiagram and $\fg_J \subseteq \fg_I$ is the corresponding Lie subalgebra, we can regard $\cC$ as a categorical representation of $U_q(\fg_J)$ by restriction.  By Theorem B we also obtain a bijection $\varphi_J:\Irr(\cC) \to \Irr(\cC)$.

\begin{namedtheorem}[C][Theorem \ref{thm:cactus} \& Theorem \ref{thm:agree}]\label{thm:C}
Let $\cC$ be a categorical representation of $U_q(\fg)$, categorifying the integrable representation $V$.  The assignment $c_J \mapsto \varphi_J$ defines an action of $C$ on $\bB_V=\Irr(\cC)$, and this  agrees with the combinatorial action arising from Sch\"utzenberger involutions.  
\end{namedtheorem}

We thus obtain the sought-after crystalisation process for braid groups:
$$
B \curvearrowright V \quad\rightsquigarrow\quad B \curvearrowright D^b(\cC) \quad\rightsquigarrow\quad C \curvearrowright \Irr(\cC),
$$
which  associates a cactus group set $\Irr(\cC)$ to the braid group representation of $B$ on $V$.  The first appearance of such a crystalisation process is in the work of the third author, where a cactus group action on $W$ is constructed \cite{LosCacti}.  It's an interesting  question to crystallise the braid group action without appealing to categorical representation theory.

Finally we remark that perversity of Rickard complexes, and more specifically the t-exactness of $\Theta_{w_0}$ on isotypic categorifications as in Theorem A, is a fruitful vantage from which to view results in algebraic combinatorics.  

For example, we show in Section \ref{sec:typeA} how to use this to easily recover theorems of Berenstein-Zelevinsky \cite{BZ96}  and Stembridge \cite{Stem96}, namely that the action of $w_0\in S_n$ 
on the Kazhdan-Lusztig basis of a Specht module of $S_n$ is governed by the evacuation operator  on standard Young tableaux.  We note that this theorem was earlier proven by Mathas in slightly different form (without explicit reference to the evacuation operator, and credited to J.J. Graham) \cite[Theorem 3.1]{Mathas}, and a similar result was shown even earlier by Lusztig in 1990 \cite[Corollary 5.9]{LusCBII}.

As another example, we use our methods to also recover Rhoades' Theorem that the Coxeter element $(1,2,\hdots,n) \in S_n$ acts on the  Kazhdan-Lusztig basis of a Specht module associated to a rectangular partition by the promotion operator.  This point of view led us to generalise Rhoades' result to arbitrary partitions \cite{GY1}, and isolate the class of permutations (the separable permutations) for which such results can hold \cite{GY2}.

\section*{Acknowledgement}
We would like to thank Sabin Cautis, Ian Grojnowski, Joel Kamnitzer,  Aaron Lauda, Andrew Mathas, Peter McNamara, Bregje Pauwels, Raph\"ael Rouquier, and Geordie Williamson  for insightful discussions.  We are grateful to two anonymous referees for their helpful comments.
I.L. is partially supported by the NSF under grant DMS-2001139.
O.Y. is supported by the Australian Research Council Grants DP180102563 and DP230100654.

\section{Background on quantum groups}
\subsection{The quantum group} In this article we work with a  simply-laced quantum group $U_q(\fg)$ of finite type.  Recall that we have an associated Cartan datum and a root datum, which consists of:
\begin{itemize}
\item A finite set $I$,
\item a symmetric bilinear form $(\cdot,\cdot)$ on $\ZZ I$ satisfying $(i,i)=2$ and $(i,j)\in \{0,-1\}$ for all $i \neq j \in I$,
\item a free $\ZZ$-module $X$, called the weight lattice, and
\item a choice of simple roots $\{ \alpha_i \}_{i\in I} \subset X$ and simple coroots $\{ h_i \}_{i\in I} \subset X^\vee=\Hom(X,\ZZ)$ satisfying $\langle h_i,\alpha_j\rangle=(i,j)$, where $\langle \cdot,\cdot\rangle:X^\vee \times X \to \ZZ$ is the natural pairing.  
\end{itemize}

The quantum group $U_q(\fg)$ is the unital, associative, $\CC(q)$ algebra generated by $E_i,F_i,K_h, (i\in I, h\in X^\vee)$ subject to relations:
\begin{enumerate}
\item $K_0=1$ and $K_hK_{h'}=K_{h+h'}$ for any $h,h'\in X^\vee$,
\item $K_hE_i=q^{\langle h,\alpha_i \rangle}E_iK_h$ for any $i\in I, h\in X^\vee$,
\item $K_hF_i=q^{-\langle h,\alpha_i \rangle}F_iK_h$ for any $i\in I, h\in X^\vee$,
\item $E_iF_j-F_jE_i=\delta_{ij}\frac{K_i-K_i^{-1}}{q-q^{-1}}$, where we set $K_i=K_{h_i}$, and
\item for all $i\neq j$,
$$
\sum_{a+b=-\langle h_i,\alpha_j \rangle+1}(-1)^aE_i^{(a)}E_j E_i^{(b)}=0 \quad \text{ and } \quad \sum_{a+b=-\langle h_i,\alpha_j \rangle+1}(-1)^aF_i^{(a)}F_j F_i^{(b)}=0,
$$
where $E_i^{(a)}=E_i^a/[a]!, F_i^{(a)}=F_i^a/[a]!$, and $[a]!=\prod_{i=1}^a \frac{q^i-q^{-i}}{q-q^{-1}}$.
\end{enumerate}

We let $a_{ij}=(i,j)$, so that $(a_{ij})_{i,j\in I}$ is a Cartan matrix.  Given $\la \in X$ we abbreviate $\la_i=\langle h_i,\la\rangle$, and let $$X_+=\{\la\in X:\la_i\geq0 \text{ for all } i\in I\}$$ be the set of dominant weights.  

Let $R\subset X$ be the root lattice, defined as the $\ZZ$-span of the simple roots, and let $R_+\subset R$ be the $\NN$-span of the simple roots.  We define the usual preorder $\succ$ on $X$ by $\lambda \succeq \mu$ if $\lambda-\mu \in R_+$.
For $\mu \in R$ let $ht(\mu)$ denote the height of $\mu$, i.e. $ht(\sum_i a_i\alpha_i)=\sum_i a_i$.

When convenient, we also view $I$ as the Dynkin diagram of $\fg$, and make reference to subdiagrams or diagram automorphisms of $I$.

\subsection{Braid group actions on integrable representations of $U_q(\fg)$.}     
Given a $U_q(\fg)$-module $V$ and $\mu \in X$ we let $V_\mu$ denote the $\mu$ weight space of $V$.  
For $\lambda \in X_+$ we let $L(\lambda)$ be the irreducible representation of $U_q(\fg)$ of highest weight $\lambda$.  Let $\Iso_\lambda(V)$ denote the $\lambda$-isotypic component of $V$.  We say that $V$ is \textbf{isotypic} if there exists $\lambda \in X_+$ such that $V=\Iso_\lambda(V)$.

The representation $L(\lambda)$ has a canonical basis, which we denote by $\bB(\lambda)$ \cite{Lusbook}.  We let $v_\lambda$ (respectively $v_\la^{low}$) denote the unique highest weight (respectively lowest weight) element of $\bB(\la)$.

Let $B=B_I$ denote the braid group of type $I$, which is generated by $\theta_i\; (i\in I)$ subject to the braid relations:
\begin{align*}
    \theta_i\theta_j &=\theta_j\theta_i, \text{ if } (i,j)=0, \text{ and }\\
    \theta_i\theta_j\theta_i&=\theta_j\theta_i\theta_j, \text{ if } (i,j)=-1.
\end{align*}
Let $W=W_I$ be the Weyl group of type $I$, which has generators $s_i\; (i\in I)$ subject to the braid relations,  and in addition the quadratic relation $s_i^2=1$.  Let $w_0\in W$  be the longest element.  Recall that $W$ acts on $X$ via $s_i\cdot \la=\la-\langle h_i, \la\rangle\alpha_i$.  We define $\tau: I\to I$ by the equality $
\alpha_{\tau(i)}=-w_0(\alpha_i)$  for any $i \in I$.

To $J \subset I$ a subdiagram, we associate $W_J \subset W$  the parabolic subgroup, $w_0^J\in W_J$  its longest element, and $\tau_J:I\to I$  the bijection given by
\begin{align*}
\alpha_{\tau_J(i)}= \begin{cases} -w_0^J(\alpha_i) &\text{ if } i \in J,\\
\alpha_i &\text{otherwise.} \end{cases}
\end{align*}

For any $w\in W$ we can consider its positive lift $\theta_w\in B$, where $\theta_w=\theta_{i_1}\cdots \theta_{i_\ell}$ and $w=s_{i_1}\cdots s_{i_\ell}$ is any reduced decomposition.

Let $V$ be an integrable representation of $U_q(\fg)$.
A fundamental structure of $V$, discovered by Lusztig, is that it admits (several) braid group symmetries, sometimes referred to as the ``quantum Weyl group actions''.  To recall this, let $\sone_\mu$ denote the projection onto the $\mu$ weight space.  For each $i\in I$ we define $\st_i:V\to V$ by:
\begin{align}
    \st_i\sone_\mu=\sum_{b-a=(\mu,\alpha_i)}(-q)^{-b}E_i^{(a)}F_i^{(b)}\sone_\mu.
\end{align}
Note that the indexing set of this sum is infinite, but the sum itself is finite on $V$.
In the notation of \cite{Lusbook}, $\st_i=\st''_{i,-1}$.  (Note that the formula given in \cite{Lusbook} is more complicated. This simpler form was initially observed in \cite{CR} in the non-quantum setting, and generalised in \cite{CKM} to the quantum setting.). The assignment $\theta_i \mapsto \st_i$ defines an action of $B$ on $V$, and so we can unambiguously write $\st_w$ for any $w\in W$.  

\section{Categorical representations of $U_q(\fg)$}
%
\subsection{Notation}
Fix a field $\kk$ of any characteristic. In this paper we will be concerned mostly with abelian $\kk$-linear categories $\cC$.  We will assume throughout that each block of $\cC$ is a finite abelian category \cite[Definition 1.8.6]{EGNO}.

Recall that a category $\cA$ is graded  if it is equipped with an auto-equivalence $\langle 1 \rangle :\cA \to \cA$ called the ``shift functor''.   We let $\langle \ell \rangle$ be the auto-equivalence obtained by applying the shift functor $\ell$ times.  We denote by $\Irr(\cA)$ the set of equivalence classes of simple objects of $\cA$ up to shift.  
 
A functor $F:\cA \to \cA'$ between graded categories is graded if it commutes with the shift functors. We denote by $[\cA]_\ZZ$ the Grothendieck group of an abelian  category $\cA$.  If $\cA$ is graded we denote by $[\cA]_{\ZZ[q,q^{-1}]}$ the quotient of $[\cA]_\ZZ \otimes_\ZZ \ZZ[q,q^{-1}]$ by the additional relation $q[M]=[M\langle-1\rangle]$. This quotient is naturally a $\ZZ[q,q^{-1}]$-module. Set 
\begin{align*}
[\cA]_\CC &=[\cA]_\ZZ\otimes_{\ZZ}\CC, \\
[\cA]_{\CC(q)} &=[\cA]_{\ZZ[q,q^{-1}]}\otimes_{\ZZ[q,q^{-1}]}\CC(q).
\end{align*}

For a $\kk$-algebra $A$, we let $A\mmod$ be the category of finitely generated $A$-modules, and if $A$ is $\ZZ$-graded, we let $A\mmod_\ZZ$ be the category of finitely generated $\ZZ$-graded modules.  These are naturally $\kk$-linear abelian categories.  

\subsection{Definition}
In this section we introduce our main objects of study: representations of the 2-quantum group on abelian categories, which we refer to as “categorical representations” of $U_q(\fg)$.  This definition of the categorified quantum groups and their categorical actions is originally due to Rouquier and Khovanov-Lauda \cite{Rou2KM, KLI,KLII}.  

In the literature, there are a number of slightly different-looking examples of 2-representations, depending not just on whether or not one is working with the Khovanov-Lauda or Rouquier 2-categories, but also depending on whether or not one is interested in categorifications of representations of $U_q(\fg)$ or of $U(\fg)$.  At the categorical level, the difference between $U_q(\fg)$ or of $U(\fg)$ arises from grading considerations; in categorifications of representations of  $U_q(\fg)$, one works with categories enriched in graded vector spaces, whereas in categorifications of representations of  $U(\fg)$ no such grading is needed.  

Fortunately, the different notions of categorical representations - and in particular relationships between the Rouquier and Khovanov-Lauda frameworks in both graded and ungraded settings, have been brought in line by work of Cautis-Lauda, who proved that in the graded setting, integrable 2-representations of the Rouquier 2-category induce 2-representations of the Khovanov-Lauda 2-category \cite{CaLa}, and by work of Brundan \cite{Brundandef}, who proved that the underlying ungraded 2-categories of Rouquier and Khovanov-Lauda are equivalent.  

The important point for us is that all our theorems, including our main results (Theorem \ref{thm:mincat} and Theorem \ref{thm:perveq}), remain true in any 2-representation of the Khovanov-Lauda-Rouqier 2-category, in either the graded or ungraded setting.
Since the compatibility of the internal grading with the braid group is of some independent combinatorial interest, we elect to keep track of the gradings in the rest of the paper, and leave it to the reader to verify that the arguments go through while ignoring the gradings and working in the setup of, e.g. Brundan \cite{Brundandef}.  To that end, we have chosen to follow the notation and conventions of Cautis-Lauda below.

A \textbf{categorical representation} of $U_q(\fg)$ consists of the following data:
\begin{itemize}
\item A family of graded abelian $\kk$-linear categories $\cC_\mu$ indexed by $\mu \in X$.  We refer to each $\cC_\mu$ as a \textbf{weight category}. 
\item Exact graded functors $\sE_i\bone_\mu:\cC_\mu\to \cC_{\mu+\alpha_i}$ and $\sF_i\bone_\mu:\cC_\mu\to \cC_{\mu-\alpha_i}$, for $i\in I$ and $\mu\in X$. We refer to $\sE_i,\sF_i$ as \textbf{Chevalley functors}. 
\item A collection of natural transformations between these functors.  We won't be using directly these natural transformations in this work, so refer the reader to \cite[Definition 1.1]{CaLa} for their definition.
\end{itemize}
This data is subject to the conditions spelled out in items (1)-(5) of \cite[Definition 1.1]{CaLa}.  We only record those that are relevant for us:
\begin{enumerate}
\item The functors $\sE_i\bone_\mu$ and $\sF_i\bone_\mu$ are biadjoint up to a specified degree shift (see \eqref{eqs:biadjs} below).  
\item The powers of  $\sE_i$ carry an action of the KLR algebra associated to $Q$, where $Q$ denotes a choice of units $(t_{ij})_{i \neq j \in I}$ in $\mathbbm{k}^\times$.  These units satisfy some restrictions which are not relevant for us.  
\item We have the following isomorphisms:
\begin{align*}
\sF_i\sE_i\bone_\mu &\cong \sE_i\sF_i\bone_\mu\oplus_{[-\langle h_i,\mu \rangle]}\bone_\mu, \text{ if  }\langle h_i,\mu \rangle \leq 0, \\
\sE_i\sF_i\bone_\mu &\cong \sF_i\sE_i\bone_\mu\oplus_{[\langle h_i,\mu \rangle]}\bone_\mu, \text{ if  }\langle h_i,\mu \rangle \geq 0,\\
\sE_i\sF_j\bone_\mu &\cong \sF_j\sE_i\bone_\mu.
\end{align*}
This notation is explained as follows: for a Laurent polynomial $f=\sum f_a q^a$, $\oplus_f A$ is a direct sum over $a\in \ZZ$ of $f_a$ copies of $A\langle a \rangle$, and $[n]:=q^{n-1}+q^{n-3}+\cdots+q^{1-n}$.
\end{enumerate}

Usually we just say that $\cC=\bigoplus_\mu \cC_\mu$ is a categorical representation of $U_q(\fg)$ (the remaining data is implicit).   Given an integrable $U_q(\fg)$-module $V$, we say that $\cC$ is a categorification of $V$ if $\cC$ is a categorical representation of $U_q(\fg)$ such that $[\cC]_{\CC(q)} \cong V$ as $U_q(\fg)$-modules.

An \textbf{additive categorification} is a categorical representation on a graded additive $\kk$-linear category $\cV$ satisfying the same conditions as above, except the Chevalley functors are of course only required to be additive.  We let $\cV^i$ be the idempotent completion of $\cV$.

Given an abelian category $\cC$, we consider the additive category $\cC\proj$ defined as the full subcategory of projective objects in $\cC$.  Note that if $\cC$ is a categorical representation of $U_q(\fg)$, then $\cC\proj$ naturally inherits the structure of an additive categorification.

As a consequence of this definition, and in particular condition (2), there exist divided power functors $\sE_i^{(r)}\bone_\lambda \subset \sE_i^r\bone_\lambda, \sF_i^{(r)}\bone_\lambda \subset \sF_i^r\bone_\lambda$ which categorify the usual divided powers on the level of the quantum group.  Again, we refer the reader to \cite{CaLa} and references therein for further details.  We note that their adjoints are related as follows:
\begin{align}\label{eqs:biadjs}
(\sE_i^{(r)}\bone_\lambda)_R &\cong \sF_i^{(r)}\bone_{\lambda+r\alpha_i}\langle r(\lambda_i+r) \rangle,\\
(\sE_i^{(r)}\bone_\lambda)_L &\cong \sF_i^{(r)}\bone_{\lambda+r\alpha_i}\langle -r(\lambda_i+r) \rangle.
\end{align}

%
%

\subsection{Crystals}

We recall the definition of a crystal. 

\begin{Definition} [\cite{Kash90}] A \textbf{$\fg$-crystal} is a finite set $\bB$ together with maps:
	$$\tilde{e}_i,\tilde{f}_i:\bB \rightarrow \bB\sqcup \{0\}, \quad \eps_i,\phi_i : \bB\rightarrow \ZZ, \quad \wt: \bB \rightarrow X$$
	for all $i \in I$, such that:
	\begin{enumerate}
		\item for any $b,b' \in \bB$, $\tilde{e}_i(b)=b'$ if and only if $b=\tilde{f}_i(b')$ ,
		\item for all $b \in \bB$, if $\tilde{e}_i(b) \in \bB$ then $\wt(\tilde{e}_i(b))=\wt(b)+\alpha_i$, and if $\tilde{f}_i(b) \in \bB$ then $\wt(\tilde{f}_i(b))=\wt(b)-\alpha_i$,
		\item for all $b \in \bB$, $\eps_i(b)=\max\{n \in \ZZ~:~\tilde{e}^n_i(b) \neq 0\}$, $\phi_i(b)=\max\{n \in \ZZ~:~\tilde{f}^n_i(b)\neq 0\}$,
		\item for all $b \in \bB$, $\phi_i(b)-\eps_i(b)=\left\langle \wt(b),h_i\right\rangle$.
	\end{enumerate}
\end{Definition}

Any $U_q(\fg)$-representation $V$ has a corresponding $\fg$-crystal $\bB=\bB_V$.  The underlying set of $\bB$ is in natural bijection with a particular basis of $V$ (the ``global crystal basis'' or ``canonical basis'').  The maps $\tilde{e}_i,\tilde{f}_i$ are related to the raising and lowering Chevalley operators; vaguely speaking they encode information about the leading terms of the Chevalley operators acting on this basis. In particular, for an integral dominant weight $\la$, the canonical basis $\bB(\la)$ of $L(\la)$ carries a natural crystal structure \cite{GL92}.

The crystal  of $V$ naturally arises via categorical representation theory. Namely, as we describe in the next proposition, if $\cC$ is a categorification of $V$, then $\Irr(\cC)$ carries a crystal structure isomorphic to $\bB_V$. This follows  from \cite[Proposition 5.20]{CR} and \cite{LV}, and is explained in detail in \cite{BDcryst}.  

\begin{Proposition}(\cite[Theorem 4.31]{BDcryst})\label{prop:cryst}
	The set $\Irr(\cC)$ together with:
	\begin{itemize}
	 \item Kashiwara operators defined as $\widetilde{\sE}_i(X)=\operatorname{soc}{\sE_i(X)}$, $\widetilde{\sF}_i(X)=\operatorname{soc}{\sF_i(X)}$ for $X \in \Irr(\cC)$, 
	 \item $\wt(X) = \mu$ for $X \in \cC_\mu$, and
	 \item $\varepsilon(X)=max\{ n \;|\; \sE_i^n(X)\neq 0\}$, and $\varphi(X)=max\{ n \;|\; \sF_i^n(X)\neq 0\}$,
	 \end{itemize}
	  is a $\fg$-crystal isomorphic to the crystal $\bB=\bB_V$.
\end{Proposition}

\subsection{Jordan-H\"older series}
A categorification of a simple representation (respectively an isotypic representation)  is called  a \textbf{simple categorification} (respectively an \textbf{isotypic categorification}).
There is a  distinguished  categorification of $L(\lambda)$ called the \textbf{minimal categorification} and denoted $\cL(\lambda)$ \cite{CR,Rou2KM,KLI,Webmerged,KK}. It  is characterized by the fact that $\cL(\lambda)_{\lambda}\cong \cL(\lambda)_{w_0(\lambda)} \cong \kk\mmod_{\ZZ}$.  We let $\kk_{low} \in \cL(\lambda)_{w_0(\lambda)}, \kk_{high}\in \cL(\lambda)_{\lambda}$ be the generators.  

%

The Jordan-H\"older Theorem for categorical representations will play an important role in our work.  
This was originally developed by Rouquier for additive categorifications \cite{Rou2KM}, and in this section we transfer these results to the abelian setting.
To set this up, recall that given finite abelian $\kk$-linear categories $\cA,\cB$ the \textbf{Deligne tensor product} $\cA \otimes_\kk \cB$  is universal for the functor assigning to every such abelian category $\cC$ the category of bilinear bifunctors $\cA \times \cB \to \cC$ right exact in both variables \cite[Definition 1.11.1]{EGNO}.  The tensor product is again a finite abelian $\kk$-linear category, and there is a bifunctor $\cA \times \cB \to \cA \otimes_\kk \cB, (X,Y) \mapsto X\otimes Y$.  This construction enjoys the following properties: 
\begin{enumerate}
\item The tensor product is unique up to unique equivalence,
\item for finite $\kk$-algebras $A,B$ we have that $(A\mmod)\otimes_\kk (B\mmod) \cong (A\otimes_\kk B)\mmod$, and
\item $\Hom_{\cA \otimes_\kk \cB}(X_1\otimes Y_1, X_2\otimes Y_2) \cong \Hom_\cA(X_1,Y_1)\otimes\Hom_\cB(X_2,Y_2)$.
\end{enumerate}

Let $\cC$ be a categorical representation, and let $\cA$ be a finite $\kk$-linear abelian category.  We can endow $\cC \otimes_\kk\cA$ with a structure of a categorical representation, by setting $(\cC \otimes_\kk\cA)_\mu=\cC_\mu \otimes_\kk\cA$, defining Chevalley functors $\sE_i\bone_\mu \otimes \bone_\cA$, etc.  If $\cC$ is a simple categorification, then clearly  $\cC \otimes_\kk\cA$ is an isotypic categorification.  Conversely, we have:  

\begin{Lemma}\label{lem:isocat}
Let $\cC$ be an isotypic categorification of type $\lambda \in X_+$.  Then there exists an abelian category $\cA$ such that $\cC \cong \cL(\lambda)\otimes_\kk \cA$.  
\end{Lemma}

\begin{proof}
Since $\cC$ is an isotypic categorification, so is $\cC\proj$.  By Rouquier's Jordan-H\"older series for additive categorifications \cite[Theorem 5.8]{Rou2KM}, there exists an additive $\kk$-linear category $\cM$ such that 
\begin{align*}
\cC\proj\cong (\cL(\lambda)\proj \otimes_\kk \cM)^i.
\end{align*}
Note that no filtration appears here since, in the notation of \cite{Rou2KM}, $(\cC\proj)_{-\lambda}^{lw}=(\cC\proj)_{-\lambda}$.

Let $\{P_i\}$ (respectively $\{Q_j\}$) be a complete list of the projective indecomposable objects of $\cL(\lambda)$ (respectively $\cM$).  Let $P=\bigoplus_{i,j}P_i\otimes Q_j$, and let $B=\End_\cC(P)^{op}$.  By Morita theory, $B\mmod\cong \cC$.  On the other hand,
\begin{align*}
B\cong \End_{\cL(\lambda)}(\bigoplus_i P_i)^{op} \otimes \End_\cM(\bigoplus_j Q_j)^{op}.
\end{align*}
Since $\cL(\lambda)\cong \End_{\cL(\lambda)}(\bigoplus_i P_i)^{op}\mmod$ we have the desired result with $\cA=\End_\cM(\bigoplus_j Q_j)^{op}\mmod$.
\end{proof}

\begin{Theorem}\label{thm:JH}
Let $\cC$ be a categorical representation of $U_q(\fg)$.  Then there exists a filtration by Serre subcategories
\begin{align}\label{eq:JHfiltration}
0=\cC_0 \subset \cC_1\subset \cdots\subset \cC_n=\cC, 
\end{align}
such that for each $i$: $\cC_i$ is a subrepresentation of $\cC$, 
$\cC_{i}/\cC_{i-1}$ is a simple categorification of type $\lambda_i\in X_+$, and the list of highest weights is weakly increasing so that $\la_i \prec \la_j \Longrightarrow i<j$. 
\end{Theorem}

\begin{proof}
The $\fg$-crystal $\Irr(\cC)$ is isomorphic to a finite direct sum of irreducible crystals $\bB(\la)$ for various $\la$, i.e. we have an isomorphism
$$
\Irr(\cC) \cong \bigoplus_{\la\in X_+} \bB(\la)^{\oplus m_\la},
$$
where $m_\la \geq 0$ and only finitely many are nonzero.

Define $M=\{ \la \in X_+ \;|\; m_\la\neq0\}$, and let $\la \in M$. We claim that there exists a highest weight simple object $L \in \cC_\la$.  Indeed, otherwise for any simple object $L\in \cC_\la$ there exists $i\in I$ such that $\sE_i(X) \neq 0$.  This implies that $\soc(\sE_i(X))\neq 0$, and by Proposition \ref{prop:cryst} we conclude that $\Irr(\cC)$ has no highest weight elements of weight $\la$,  a contradiction.  

Now take $\la \in M$  which is minimal with respect to $\preceq$, and let $L \in \cC_\la$ be a highest weight object.
 Let $\cC_1$ be the Serre subcategory of $\cC$ generated by objects
\begin{align}\label{eq:Serregens}
\{\sF_{i_1}\cdots \sF_{i_\ell}(L)\;|\; i_j\in I, \ell \geq 0\}.
\end{align}
By the exactness and bi-adjunction of the Chevalley functors, $\cC_1$ is  a subrepresentation of $\cC$.  Moreover it categorifies $L(\lambda)$.  Indeed, by our choice of $\lambda$ there cannot be any highest weight objects with weight $\prec\lambda$ occuring in $\cC_1$, and by construction the only simple object in  $(\cC_1)_\lambda$ is $L$.  

Next consider the categorical representation $\cC/\cC_1$ and repeat this construction.  This produces a Serre subcategory $\cC_2' \subset \cC/\cC_1$ which is again a simple categorification.  Let $\pi:\cC \to \cC/\cC_1$ be the natural quotient functor, and define $\cC_2=\pi^{-1}(\cC_2')$.  Clearly, we have that $\cC_1\subset \cC_2$, $\cC_2$ is Serre, it is a subrepresentation, and $\cC_2/\cC_1 \cong \cC_2'$ is a simple categorification.  

Iterating this process produces a filtration of $\cC$ such that each composition factor is a simple categorification, and the highest weights of the subquotients are weakly increasing.
\end{proof}

\begin{Remark}\label{rem:isofilts}
Note that the construction in the proof of Theorem \ref{thm:JH} can produce also an isotypic filtration with similar properties.  Namely, if $\la_1,\hdots,\la_N$ is a list of the \textit{distinct} isotypic types appearing in $[\cC]_{\CC(q)}$, and we choose any ordering of this list so that $\lambda_i \prec \lambda_j \Longrightarrow i <j$, then there is a filtration $0=\cC'_0 \subset \cC'_1\subset \cdots\subset \cC'_N=\cC$ such that $\cC'_{k}/\cC'_{k-1}$ categorifies the isotypic component of $[\cC]_{\CC(q)}$ of highest weight $\lambda_k$.  To construct this isotypic filtration, consider the  Jordan-H\"older filtration  from the theorem.  From the proof of Theorem \ref{thm:JH}, it's easy to see that one can ensure that the subquotients which categorify the same simple representations appear in sequence.  Assuming then that our Jordan-H\"older filtration satisfies this property, a coarsening of it is  the desired isotypic filtration.
\end{Remark}

\begin{Remark}\label{rem:JHfromCrystal}
Note that one can read off the isotypic filtration of $\cC$  from the crystal structure on $\Irr(\cC)$.  Indeed, suppose that $\Irr(\cC)$ decomposes into  components 
$$
\Irr(\cC)=\bX(\la_1) \sqcup \cdots \sqcup \bX(\la_N),
$$
where $\la_1,\hdots,\la_N \in X_+$ are distinct dominant integral weights, and $\bX(\la_i)$ is a disjoint union of copies of $\bB(\la_i)$.  Further, we arrange the weights  as above so that  $\la_i \prec \la_j \Longrightarrow i < j$.  Let $\cC_i$ be the Serre subcategory of $\cC$ generated by simple objects $L$ such that $[L] \in \bX(\la_j)$, where $j \leq i$.  Then it follows from Remark \ref{rem:isofilts} that $\{0\} \subset \cC_1 \subset \cC_2 \subset \cdots \subset \cC_N$ is an isotypic filtration of $\cC$.
\end{Remark}

\subsection{The categorical braid group action}
Let $\cC$ be a categorical representation of $U_q(\fg), \mu \in X$ and $i\in I$.     
We define a complex of functors $\Theta_i\bone_\mu$, supported in nonpositive cohomological degrees, where for $r\geq0$ the $-r$ component is
\begin{align*}
(\Theta_i\bone_\mu)^{-r}=\begin{cases}\sE_i^{(-\mu_i+r)}\sF_i^{(r)}\bone_\mu\langle -s \rangle &\text{ if } \mu_i\leq0, \\
\sF_i^{(\mu_i+r)}\sE_i^{(r)}\bone_\mu\langle -r \rangle &\text{ if } \mu_i\geq0.
\end{cases}
\end{align*}
The differential $d^r:(\Theta_i\bone_\mu)^{-r} \to (\Theta_i\bone_\mu)^{-r+1}$ is defined using the counits of the bi-adjunctions relating $\sE_i$ and $\sF_i$ (see \cite[Section 4]{Cauclasp} for details).  This produces a functor $\Theta_i\bone_\mu:D^b(\cC_\mu)\to D^b(\cC_{s_i(\mu)})$, which following Chuang and Rouquier we call the \textbf{Rickard complex}.

It's straightforward to verify that the Rickard complex $\Theta_i\bone_\mu$ categorifies Lusztig's braid group operators $\st_i\sone_\mu$ (\cite[Section 2]{Cauclasp}).
On the level of categories we have the following two theorems of Chuang-Rouquier and Cautis-Kamnitzer, which are the fundamental results about  Rickard complexes.  Note that the latter theorem was conjectured in \cite[Conjecture 5.19]{Rou2KM}.

\begin{Theorem}\cite[Theorem 6.4]{CR}\label{thm:CR}
For any $\mu,i$, $\Theta_i\bone_\mu:D^b(\cC_\mu)\to D^b(\cC_{s_i(\mu)})$ is an equivalence of triangulated categories.
\end{Theorem}

\begin{Theorem}\cite[Theorem 6.3]{CK3}\label{thm:CK}
The Rickard complexes satisfy the braid relations:
\begin{align*}
\Theta_i\Theta_j\bone_\mu &\cong \Theta_j\Theta_i\bone_\mu \text{ if } (i,j)=0,\\
\Theta_i\Theta_j\Theta_i\bone_\mu &\cong \Theta_j\Theta_i\Theta_j\bone_\mu \text{ if } (i,j)=-1,
\end{align*}
thereby defining a weak action of $B$ on $D^b(\cC)$.
\end{Theorem}

This action is ``weak'' since we don't make any claim on the canonicity of the functorial isomorphisms.  Nevertheless, for $w \in W$ we define $\Theta_w\bone_\mu := \Theta_{i_1}\circ\cdots\circ\Theta_{i_\ell}\bone_\mu$, where $w=s_{i_1}\cdots s_{i_\ell}$ is a reduced expression.  Thus  $\Theta_w\bone_\mu$ is defined up to isomorphism, but not canonical isomorphism.  Luckily, everything we do in this paper only requires $\Theta_w\bone_\mu$ to be defined up to isomorphism.

As a consequence of their proof of Theorem \ref{thm:CR}, Chuang and Rouquier show that the inverse  of $\Theta_i\bone_\mu$ is its right adjoint.  We denote this functor by $\Theta_i'\bone_{\mu}:D^b(\cC_\mu)\to D^b(\cC_{s_i(\mu)})$, so that $\Theta_i\Theta_i'\bone_\mu \cong \Theta'_i\Theta_i\bone_\mu \cong \bone_\mu$.  As a complex of functors, $\Theta_i'\bone_\mu$ is supported in nonnegative cohomological degrees, where for $s\geq0$ the $s$ component is
\begin{align*}
(\Theta_i'\bone_\mu)^s =\begin{cases}
\sE_i^{(s)}\sF_i^{(\mu_i+s)}\bone_\mu\langle s(-2\mu_i-2s+1) \rangle &\text{ if } \mu_i \geq 0, \\
\sF_i^{(s)}\sE_i^{(-\mu_i+s)}\bone_\mu\langle s(-2\mu_i+2s+1) \rangle&\text{ if } \mu_i \leq 0.
\end{cases}
\end{align*}

\section{Perverse equivalences}
\subsection{General definition}\label{sec:pervdef}
Let $\cT$ be a triangulated category with shift functor $[1]:\cT \rightarrow \cT$. In the cases of most interest to us, $\cT$ is a subcategory of a derived category, in which case $[1]$ is the homological shift functor.
Suppose $\cT$ has a t-structure $t=(\cT^{\leq0},\cT^{\geq0})$, with heart $\cT^\heartsuit=\cT^{\leq0} \cap \cT^{\geq0}$ \cite{BBD}.  Recall that a triangulated functor $F:\cT \to \cS$ between triangulated categories with $t$-structure is \textbf{$t$-exact} if $F(\cT^{\leq0})\subseteq \cS^{\leq0}$ and $F(\cT^{\geq0})\subseteq \cS^{\geq0}$.  We let $F[p]: \cT \to \cS$ denote the pre-composition of $F$ with the $p$-shift $[p]$.

Now let $\cS \subset \cT$ be a thick triangulated subcategory, and consider the quotient functor $Q:\cT \to \cT/\cS$.  Following \cite{CRperv}, we say that $t$ is \textbf{compatible} with $\cS$ if $t_{\cT/\cS}=(Q(\cT^{\leq0}),Q(\cT^{\geq0}))$ is a t-structure on $\cT/\cS$.  By \cite[Lemmas 3.3 \& 3.9]{CRperv}, if $t$ is compatible with $\cS$ then $(\cT/\cS)^\heartsuit=\cT^\heartsuit/\cT^\heartsuit\cap \cS$, and  $t_\cS=(\cS\cap\cT^{\leq0},\cS\cap\cT^{\geq0})$ is a t-structure on $\cS$ such that $\cS^\heartsuit=\cT^\heartsuit\cap \cS$.

Now suppose that $\cT,\cT'$ are two triangulated categories with t-structures $t,t'$.  Suppose further that we have filtrations by thick triangulated subcategories:
\begin{align*}
0 \subset \cT_0 \subset \cT_1\subset \cdots\subset \cT_r=\cT ,\quad 0 \subset \cT_0' \subset \cT_1'\subset \cdots\subset \cT_r'=\cT', 
\end{align*}
such that for every $i$, $t$ is compatible with $\cT_i$ and $t'$ is compatible with $\cT'_i$. By \cite[Lemma 3.11]{CRperv}, $t_{\cT_i}$ is also compatible with $\cT_{i-1}$, and hence $\cT_i/\cT_{i-1}$ inherits a natural t-structure.  Let $p: \{0,\hdots,r\} \rightarrow \ZZ$.  The data $(\cT_\bullet, \cT'_\bullet, p)$ is termed a \textbf{perversity triple}.  

Although Chuang and Rouquier didn't formulate perverse equivalences for graded categories, it is straightforward to extend their definitions to this setting.

\begin{Definition}\label{def:PE}
A graded equivalence of graded triangulated categories $\sF:\cT\to\cT'$ is a \textbf{(graded) perverse equivalence} with respect to $(\cT_\bullet, \cT'_\bullet, p)$ if for every $i$, 
\begin{enumerate}
\item $\sF(\cT_i)= \cT_i'$, and
\item the induced equivalence $\sF[-p(i)]:\cT_i/\cT_{i-1}\to \cT_i'/\cT_{i-1}'$ is t-exact.
\end{enumerate}
\end{Definition}
For brevity, we say $\sF$ is \textbf{perverse} if it is a graded perverse equivalence with respect to some perversity datum.  Since we will be working exclusively in the graded setting, a perverse equivalence for us will always mean a graded perverse equivalence.    

A perverse equivalence $\sF:\cT\to\cT'$  induces a bijection $\varphi_\sF:\Irr(\cT^\heartsuit) \to \Irr(\cT^{\prime\heartsuit})$.  Indeed, by (2) $\sF[-p(i)]$ induces a bijection $\Irr(\cT_i^\heartsuit)\setminus \Irr(\cT_{i-1}^\heartsuit) \to \Irr(\cT_i^{\prime\heartsuit})\setminus \Irr(\cT_{i-1}^{\prime\heartsuit})$, and these yield $\varphi_\sF$.

Although the construction of $\varphi_\sF$ depends on a choice of perversity triple, the resulting bijection does not when $\cT^\heartsuit,\cT^{\prime\heartsuit}$ have finitely many simple objects.  This follows from the following lemma.

\begin{Lemma}[\cite{LosCacti}, Lemma 2.4]\label{lem:LosLem}
 Suppose that $\cT^\heartsuit,\cT^{\prime\heartsuit}$ have finitely many simple objects, and for $i=1,2$ let $\sF_i:\cT \to \cT'$ be a  perverse equivalence with respect to the perversity datum $(\cT_{i,\bullet}, \cT'_{i,\bullet}, p_i)$.  If the induced maps $[\sF_1],[\sF_2]:[\cT]_\ZZ \to [\cT']_\ZZ$ coincide then $\varphi_{\sF_1}=\varphi_{\sF_2}$.
 \end{Lemma}
 
 \begin{Corollary}
 Suppose that $\cT^\heartsuit,\cT^{\prime\heartsuit}$ have finitely many simple objects, and  $\sF:\cT \to \cT'$ is perverse.  Then $\varphi_\sF$ is independent of the choice of perversity triple.
 \end{Corollary}
 
 \begin{proof}
Suppose $\sF$ is a graded perverse equivalence with respect to two choices of perversity triples $(\cT_{i,\bullet}, \cT'_{i,\bullet}, p_i)$, $i=1,2$.  Now apply the lemma.
 \end{proof}
 
 The proofs of the following lemmas are straightforward.
 
 \begin{Lemma}\label{lem:pervshift}
 Suppose $\sF:\cT \rightarrow \cT'$ is perverse.  Then for any $\ell\in\ZZ$, $\sF[\ell]$ is also perverse and $\varphi_{\sF[\ell]}=\varphi_\sF$.
 \end{Lemma}
 
 \begin{Lemma}\label{lem:pervcomp}
 Suppose $\sF:\cT \rightarrow \cT'$ is a  perverse equivalence with respect to $(\cT_\bullet, \cT'_\bullet, p)$, and $\sG:\cT' \rightarrow \cT''$ is a  perverse equivalence with respect to $(\cT'_\bullet, \cT''_\bullet, q)$.  Then $\sG\circ \sF$ is a  perverse equivalence with respect to $(\cT_\bullet, \cT''_\bullet, p+q)$, and $\varphi_{\sG\circ \sF}=\varphi_\sG\circ\varphi_\sF$.
 \end{Lemma}
 
\begin{Lemma}\label{lem:quot}
Let $\cT,\cT'$ be triangulated with t-structures $t,t'$, and let $\cS\subset \cT, \cS'\subset \cT'$ be thick triangulated subcategories such that $t$ is compatible with $\cS$ and $t'$ is compatible with $\cS'$.  Suppose further that $\sF:\cT\to\cT'$ is a perverse equivalence with respect to $(\cT_\bullet, \cT'_\bullet, p)$, and $\sF(\cS)=\cS'$.  

Define $\cS_i=\cS\cap \cT_i, \cS_i'=\cS'\cap \cT_i'$ and  $(\cT/\cS)_i=Q(\cT_i), (\cT'/\cS')_i=Q'(\cT_i')$.  Let 
$\sG:\cS\to\cS'$ and $\sH:\cT/\cS \to \cT'/\cS'$ be the induced equivalences.  Then:
\begin{enumerate}
\item $\sG$ is a perverse equivalence with respect to $(\cS_\bullet, \cS'_\bullet,p)$, and $\varphi_{\sG}=\varphi_{\sF}|_{\Irr{\cS^\heartsuit}}$.
\item $\sH$ is a perverse equivalence with respect to $((\cT/\cS)_\bullet,(\cT'/\cS')_\bullet,p)$, and
$
\varphi_{\sH}=\varphi_{\sF}|_{\Irr{\cT^\heartsuit} \setminus \Irr{\cS^\heartsuit}}.
$
\end{enumerate}
\end{Lemma}

\begin{Lemma}[\cite{LosCacti}, Lemma 2.4]\label{lem:texactperv}
Suppose $\sF:\cT \rightarrow \cT'$ is perverse, and $\sG$ (respectively $\sG'$) is an autoequivalence of $\cT$ (respectively $\cT'$) which is t-exact up to shift.  Then $\sG'\circ \sF \circ \sG$ is perverse, and $\varphi_{\sG'\circ \sF \circ \sG}=\varphi_{\sG'}\circ\varphi_{\sF}\circ\varphi_{\sG'}$.  
\end{Lemma}

Note that in \cite[Lemma 2.4]{LosCacti}  is stated for functors which are t-exact.  Our formulation for functors which are t-exact up to shift follows by Lemma \ref{lem:pervshift}.

\subsection{Derived categories of graded abelian categories}
We now specialise to the case of derived categories.  We recall that if $\cA$ is an abelian category, then the bounded derived category $D^b(\cA)$ has a standard t-structure $(D^b(\cA)^{\leq0},D^b(\cA)^{\geq0})$ whose heart is $\cA$.  

Given a $\cB \subset \cA$ a Serre subcategory, we let $D_\cB^b(\cA) \subset D^b(\cA)$ denote the thick subcategory consisting of complexes with cohomology supported in $\cB$. The category $D_\cB^b(\cA)$ inherits a natural t-structure from the standard t-structure on $D^b(\cA)$: $D_\cB^b(\cA)^{\leq0}=D_\cB^b(\cA) \cap D^b(\cA)^{\leq0}$ and $D_\cB^b(\cA)^{\geq0}=D_\cB^b(\cA) \cap D^b(\cA)^{\geq0}$.  The heart of the t-structure on 
$D_\cB^b(\cA)$ is $\cB$.  Moreover, if $\cC \subset \cB$ is another Serre subcategory then 
the $t$-structure on $D_\cB^b(\cA)$ is compatible with $D_\cC^b(\cA)$.  In particular, the quotient  $D_\cB^b(\cA)/D_\cC^b(\cA)$ inherits a natural t-structure whose heart is $\cB/\cC$.

For the remainder of this section let $\cA,\cA'$ be graded abelian categories.  In the setting of derived categories of graded abelian categories, a perverse equivalence can be packaged as follows.  We can encode a perversity triple $(\mathcal{A}_\bullet, \mathcal{A}'_\bullet, p)$ using filtrations on the abelian categories: $\cA_\bullet$ and $\cA'_\bullet$ are filtrations by shift-invariant Serre subcategories:
\begin{align*}
	0 = \mathcal{A}_{-1} \subset \mathcal{A}_0 \subset \mathcal{A}_1 \subset \hdots \subset \mathcal{A}_r=\mathcal{A}, \quad
	0 = \mathcal{A}'_{-1} \subset \mathcal{A}'_0 \subset \mathcal{A}'_1 \subset \hdots \subset \mathcal{A}'_r=\mathcal{A}'.
	\end{align*}
Then a graded equivalence $\sF:D^b(\mathcal{A}) \to D^b(\mathcal{A}')$ is a perverse with respect to $(\mathcal{A}_\bullet, \mathcal{A}'_\bullet, p)$ if conditions (1) and (2) of Definition \ref{def:PE} hold for $\cT_i=D^b_{\mathcal{A}_i}(\mathcal{A})$ and $\cT_i'=D^b_{\mathcal{A}'_i}(\mathcal{A}')$


 As above, a graded perverse equivalence $\sF:D^b(\mathcal{A}) \to D^b(\mathcal{A}')$ induces a bijection $\varphi_\sF:\Irr(\cA)\to\Irr(\cA')$.  
 
The following standard lemma will be useful in the proof of our main result.  

\begin{Lemma}\label{lem:trianginv}
Let $\cA,\cA'$ be abelian categories, and $\cB,\cB'$ Serre subcategories.  Let $a\leq b$ be integers, and $\sF_i:\cA \to \cA'$ be exact functors for $a \leq i \leq b$.  Suppose these functors  fit into a  complex $\sF=(\sF_a \to \sF_{a+1} \to \cdots \to \sF_{b} )$, defining a functor 
$$
\sF:D^b(\cA) \to D^b(\cA').
$$
If $\sF_i(\cB)\subset \cB'$ for all $a \leq i \leq b$, then $\sF(D_\cB^b(\cA))\subseteq D_{\cB'}^b(\cA')$.
\end{Lemma}

\section{Some commutation relations}\label{sec:commrels}
We fix throughout a categorical representation $\cC$ of  $U_q(\fg)$.
Recall that $w_0\in W$ is the longest word, and let $w_0=s_{i_1}s_{i_2} \cdots s_{i_r}$ be a reduced expression.
We consider the composition of Rickard complexes which categorifies the positive lift in $B$ of $w_0$:
$$
\Theta_{w_0}\bone_\lambda=\Theta_{i_1} \cdots \Theta_{i_r}\bone_\mu:D^b(\cC_\lambda) \to D^b(\cC_{w_0(\lambda)}).
$$
In preparation for the proofs our main results in the next section, we prove some commutation relations between $\Theta_{w_0}$ and the Chevalley functors.

\subsection{Cautis' relations}
To begin, we recall some relations of Cautis (building on work with Kamnitzer \cite{CK3}).  Although they are stated only for type A,  their proofs apply to any simply-laced Lie algebra.

\begin{Lemma}[Lemma 4.6, \cite{Cauclasp}]\label{lem:Cau0}
For any $i\in I, \lambda \in X$ we have the following relations:
\begin{align*}
\Theta_i\sE_i\bone_\lambda &\cong \sF_i\Theta_i\bone_\lambda[1]\langle \lambda_i \rangle, 
\\
\Theta_i\sF_i\bone_\lambda &\cong \sE_i\Theta_i\bone_\lambda[1]\langle -\lambda_i \rangle. 
\end{align*}
\end{Lemma}

\begin{Remark}
The careful reader will notice that actually Cautis proves Lemma \ref{lem:Cau0} under certain conditions on $\lambda$.  For instance, the first relation is only proven in the case when $\lambda_i \leq 0$.  To deduce the general case from this, one can rewrite the relation
as $\sE_i \Theta_i^{-1}\cong \Theta_i^{-1}\sF_i[1]\langle \lambda_i \rangle$.  Now recall that there is an anti-automorphism $\tilde{\sigma}$ on the $\fsl_2$ 2-category which on objects maps $n \mapsto -n$ \cite[Section 5.6]{Laucq}.  This anti-automorphism maps $\Theta_i^{-1}$ to $\Theta_i$, and hence applying it to the relation above we deduce the desired relation in the case when $\lambda_i \geq 0$.

Alternatively, in a recent preprint Vera proves a version of the relation between the Rickard complex and Chevalley functors in the (bounded homotopy category of the) $\fsl_2$ 2-category, which of course implies it also in any 2-representation \cite{Vera}.
\end{Remark}

%

Next we recall the categorical analogues of commutators $[E_i,E_j]$ acting on representations of $U_q(\fg)$.  Given nodes $i,j\in I$ such that $(i,j)=-1$ and  $\lambda \in X$, define complexes of functors 
\begin{align*}
\sE_{ij}\bone_\lambda:D^b(\cC_\lambda) \to D^b(\cC_{\lambda+\alpha_i+\alpha_j}), \quad \sE_{ij}\bone_\lambda &=\sE_i\sE_j\bone_\lambda\langle-1\rangle \to \sE_j\sE_i\bone_\lambda, \\
\sF_{ij}\bone_\lambda: D^b(\cC_\lambda) \to D^b(\cC_{\lambda-\alpha_i-\alpha_j}), \quad \sF_{ij}\bone_\lambda &=\sF_i\sF_j\bone_\lambda \to \sF_j\sF_i\bone_\lambda\langle 1 \rangle.
\end{align*}
In both instances the differential  is given by the element $T_{ij}$ arising from the KLR algebra, and the left term of the complex is in homological degree zero \cite{Cauclasp}.

\begin{Lemma}[Lemma 5.2, \cite{Cauclasp}]\label{lem:Cau} Let $i,j\in I,  \lambda \in X$ and suppose $(i,j)=-1$.  We have the following isomorphisms:
\begin{align*}
\sE_{ij} \Theta_i \bone_\lambda &\cong \begin{cases} \Theta_i\sE_j &\text{ if } \lambda_i>0, \\
\Theta_i\sE_j[1]\langle -1 \rangle &\text{ if } \lambda_i\leq 0 \end{cases} \\
\sF_{ij} \Theta_i \bone_\lambda &\cong \begin{cases} \Theta_i\sF_j &\text{ if } \lambda_i\geq0, \\
\Theta_i\sF_j[-1]\langle 1 \rangle &\text{ if } \lambda_i< 0 \end{cases} \\
\bone_\lambda\Theta_j\sE_{ij} &\cong \begin{cases} \sE_i\Theta_j &\text{ if } \lambda_j<0, \\
\sE_i\Theta_j[1]\langle -1 \rangle &\text{ if } \lambda_j\geq 0 \end{cases} \\
\bone_\lambda\Theta_j\sF_{ij} &\cong \begin{cases} \sF_i\Theta_j &\text{ if } \lambda_j\leq0, \\
\sF_i\Theta_j[-1]\langle 1 \rangle &\text{ if } \lambda_j> 0 \end{cases} 
\end{align*}

\end{Lemma}


%
%
\subsection{Marked words}
We now introduce a combinatorial set-up which we'll use to prove  Proposition \ref{prop:mainrel} below.  A \textbf{marked word} is a word in the elements of $I$ with one letter marked: $\ba=(i_1,i_2,...,\ul{i_\ell},...,i_n)$.  From $\ba$ we can define a functor and an element of $W$: 
\begin{align*}
\Phi(\ba)&=\Theta_{i_1}\cdots\Theta_{i_{\ell-1}}\sF_{i_\ell}\Theta_{i_{\ell1}}\cdots \Theta_{i_n}\bone_\lambda, \\
w(\ba)&=s_{i_1}\cdots s_{i_n}.
\end{align*}
Note that unlike $\Phi(\ba)$, $w(\ba)$ forgets the location of the marked letter.  We say that $\ba$ is \textbf{reduced}, if the corresponding unmarked word is a reduced expression for $w(\ba)$. 

We will apply braid relations to marked words.  Away from the marked letter these operate as usual, and at the marked letter we have:
\begin{align}
\label{rel1}(\hdots,k,\ul{\ell},\hdots) &\leftrightarrow (\hdots,\ul{\ell},k,\hdots),\quad \text{ if } (k,\ell)=0 \text{ and,}\\
\label{rel2}(\hdots,\ell,k,\ul{\ell},\hdots) &\leftrightarrow (\hdots,\ul{k},\ell,k,\hdots),\quad \text{ if } (k,\ell)=-1.
\end{align}
For marked words $\ba,\bb$ we write $\ba \sim \bb$ if they are related by a sequence of braid relations.  
\begin{Lemma}
Let $\ba,\bb$ be marked words which differ by a single braid relation.  Then  there exists $k\in \{0,\pm 1\}$ such that $\Phi(\ba) \cong \Phi(\bb)[k]\langle -k \rangle$.
\end{Lemma}
\begin{proof}
If the relation doesn't involve the marked letter then $\Phi(\ba)\cong \Phi(\bb)$ since Rickard complexes satisfy the braid relations \cite[Theorem 2.10]{CK3}.  Suppose then that the relation does involve the marked letter.
If $(j,\ell)=0$ the result follows from the fact that $\Theta_j\sF_\ell \cong \sF_\ell \Theta_j$.  Otherwise $(j,\ell)=-1$.  Set $\mu=s_\ell s_j(\lambda)-\alpha_j$.  Applying Lemma \ref{lem:Cau} (twice) we deduce that
\begin{align*}
\bone_\mu\Theta_\ell\Theta_j \sF_\ell \bone_\lambda = \Theta_\ell\Theta_j \sF_\ell \bone_\lambda &\cong
\begin{cases}
\Theta_\ell \sF_{j\ell}\Theta_j\bone_\lambda, &\text{ if } \lambda_j\geq 0,\\
\Theta_\ell \sF_{j\ell}\Theta_j\bone_\lambda[1]\langle -1 \rangle, &\text{ if } \lambda_j< 0.
\end{cases}\\
&\cong
\begin{cases}
\sF_j\Theta_\ell\Theta_j\bone_\lambda[-1]\langle 1 \rangle &\text{ if } \lambda_j \geq0, \mu_\ell > 0,\\
\sF_j\Theta_\ell\Theta_j\bone_\lambda[1]\langle -1 \rangle &\text{ if } \lambda_j <0, \mu_\ell \leq 0,\\
\sF_j\Theta_\ell\Theta_j\bone_\lambda &\text{ otherwise.}
\end{cases}
\end{align*}
Hence the result follows.
\end{proof}

\begin{Corollary}\label{lem:markwrds1}
Let $\ba,\bb$ be marked words such that $\ba \sim \bb$.  Then there exists an integer $k$ such that $\Phi(\ba) \cong \Phi(\bb)[k]\langle -k \rangle$.
\end{Corollary}

\begin{Lemma}\label{lem:markwrds2}
Let $\ba=(i_1,...,i_n,\ul{\ell})$ and $\bb=(\ul{\ell'},i_1,...,i_n)$ be reduced marked words such that $w(\ba)=w(\bb)$.  Then $\ba \sim \bb$.  
\end{Lemma}

\begin{proof}
We prove the claim by induction on $n$.  Since $w(\ba)=w(\bb)$ there is a Matsumoto sequence of length $M$ relating the unmarked words  $\fa=(i_1,...,i_n,\ell)$ and $\fb=(\ell',i_1,...,i_n)$.  Let $\fa_0=\fa \to \fa_1 \to\cdots\to \fa_M=\fb$ denote the 
resulting sequence of unmarked words, starting at $\fa$ and ending $\fb$.  Let $\fa_{r+1}$ be the first word in this sequence whose last entry is not $\ell$.  In other words, the first $r$ steps in the sequence do not involve the last entry of $\fa$, but the $(r+1)^{st}$ step does.  

We will now consider  the same sequence of steps, but applied to the marked words.  We then obtain a sequence of marked words $\ba_0=\ba \to \ba_1 \to \cdots\;$.  Further, we know that $\ba_r=(j_1,\dots,j_n,\ul{\ell})$ for some $j_1,\hdots,j_n$, and the next step involves the marked letter.  

If the $(r+1)^{st}$ step is an application of (\ref{rel2}) then $j_{n-1}=\ell$ and the step is:
\begin{align*}
\ba_r=(j_1,\hdots,j_{n-2},\ell,j_n,\ul{\ell}) \to \ba_{r+1}=(j_1,\hdots,j_{n-2},\ul{j_n},\ell,j_n).
\end{align*}
Let $\bb':=(\ul{\ell'},j_1,...,j_{n-2},\ell,j_n)$, i.e. $\bb'$ is obtained from $\bb$ by applying the first $r$ steps to its (last $n$) entries as were used for the (first $n$) entries of $\ba$. Note that $\ba_{r+1}$ and $\bb'$ are reduced marked words such that $w(\ba_{r+1})=w(\bb')$, and both end in $(\ell,j_n)$.    Hence, if we delete these last two entries we obtain two reduced marked words $\ba''=(j_1,...,j_{n-2},\ul{j_n})$ and $\bb''=(\ul{\ell'},j_1,...,j_{n-2})$ such that $w(\ba'')=w(\bb'')$.  
By induction $\ba'' \sim \bb''$ and therefore $\ba_{r+1} \sim \bb'$.  Since clearly $\bb' \sim \bb$, we have our desired result:
$$
\ba \sim \ba_r \sim \ba_{r+1} \sim \bb' \sim \bb.
$$

If  the $(r+1)^{st}$ step  is an application of (\ref{rel1}) then the step is:
\begin{align*}
\ba_r=(j_1,\hdots,j_{n-1},j_n,\ul{\ell}) \to \ba_{r+1}=(j_1,\hdots,j_{n-1},\ul{\ell},j_n).
\end{align*}
 Note that  $\ba_{r+1}$ and $\bb':=(\ul{\ell'},j_1,...,j_{n-1},j_n)$ are reduced marked words such that $w(\ba_{r+1})=w(\bb')$, and both end in $j_n$.  Hence we can delete this last entry and apply a similar analysis as above.

\end{proof}

\subsection{The relation between $\Theta_{w_0}$ and Chevalley functors}

We now have the machinery in place to prove our main relation.  

\begin{Proposition}\label{prop:mainrel}
For any $i\in I, \lambda \in X$ we have the following relations:
\begin{align*}
\Theta_{w_0}\sE_i\bone_\lambda &\cong \sF_{\tau(i)}\Theta_{w_0}\bone_\lambda[1]\langle \lambda_i \rangle, 
\\
\Theta_{w_0}\sF_i\bone_\lambda &\cong \sE_{\tau(i)}\Theta_{w_0}\bone_\lambda[1]\langle -\lambda_i \rangle. 
\end{align*}
\end{Proposition}

\begin{proof}
We'll prove the first relation, the second being entirely analogous.  
For two functors $\sF,\sG$ we write $\sF \equiv \sG$ if there exist integers $\ell,k$ such that $\sF \cong \sG[\ell]\langle k \rangle$.  

We first show that $\Theta_{w_0}\sE_i \equiv \sF_{\tau(i)}\Theta_{w_0}$ by induction on the rank of $\fg$.   The base case, when $\fg=\fsl_2$,  follows from Lemma \ref{lem:Cau0}.  For the inductive step let $J \subset I$ be a strict subdiagram containing $i$.  Recall the bijection $\tau_J:I\to I$ induced by the longest element $w_0^J \in W_J$.  Let $u=w_0(w_0^J)^{-1}$ and let $u=s_{i_1}\cdots s_{i_n}$ be a reduced expression.  Define two marked words:
\begin{align*}
\ba&=(i_1,\hdots,i_n,\ul{\tau_J(i)}),\\
\bb&=(\ul{\tau(i)},i_1,\hdots,i_n).
\end{align*}
Note that $w(\ba)=w(\bb)$.  By the inductive hypothesis we have $\Theta_{w_0^J}\sE_i 
\equiv  \sF_{\tau_J(i)}\Theta_{w_0^J}$, and therefore
\begin{align*}
\Theta_{w_0}\sE_i \equiv \Theta_u\Theta_{w_0^J}\sE_i 
\equiv  \Theta_u\sF_{\tau_J(i)}\Theta_{w_0^J} \equiv \Phi(\ba)\Theta_{w_0^J}.
\end{align*}
Note that $\ba,\bb$ satisfy the hypothesis of Lemma \ref{lem:markwrds2}, so by Lemmas \ref{lem:markwrds1} and \ref{lem:markwrds2}
we have that $\Phi(\ba)\equiv\Phi(\bb)$, and hence
$\Theta_{w_0}\sE_i \equiv \Phi(\bb)\Theta_{w_0^J} \equiv \sF_{\tau(i)}\Theta_{w_0}.$

We now know there exist integers $k,\ell$ such that $\Theta_{w_0}\sE_i\bone_\lambda \cong \sF_{\tau(i)}\Theta_{w_0}\bone_\lambda[\ell]\langle k \rangle$, and it remains to show that $\ell=1, k=\lambda_i$.  Let $u=w_0s_i$ and let $u=s_{i_1}\cdots s_{i_n}$ be a reduced expression.  Define two marked words:
\begin{align*}
\ba&=(i_1,\hdots,i_n,\ul{i}),\\
\bb&=(\ul{\tau(i)},i_1,\hdots,i_n).
\end{align*}
Note that $\ba,\bb$ satisfy the hypothesis of Lemma \ref{lem:markwrds2}, so by Lemmas \ref{lem:markwrds1} and \ref{lem:markwrds2} there exists an integer $m$ such that $\Phi(\ba)\cong \Phi(\bb)[m]\langle -m \rangle$.  Hence we have that
\begin{align*}
\Theta_{w_0} \sE_i \bone_\lambda &\cong \Theta_u \Theta_i \sE_i\bone_\lambda \\
& \cong \Theta_u \sF_i \Theta_i\bone_\lambda [1]\langle \lambda_i \rangle \\
& \cong \sF_{\tau(i)}\Theta_u \Theta_i \bone_\lambda[m+1]\langle -m+\lambda_i \rangle \\
& \cong \sF_{\tau(i)}\Theta_{w_0} \bone_\lambda[m+1]\langle -m+\lambda_i \rangle
\end{align*}
showing that $\ell+k=1+\lambda_i$.  

On the other hand, we can deduce $k$ by inspecting the relation on the level of Grothendieck groups.  Namely, by \cite[Lemma 5.4]{KT}, we have that 
$$
\st_{w_0}E_i\sone_\lambda = -q^{-\lambda_i}F_{\tau(i)}\st_{w_0}\sone_\lambda,
$$
showing that $k=\lambda_i$.
\end{proof}

\section{On t-exactness and perversity of $\Theta_{w_0}$}

In this section we will state and prove the central results of the paper.  We fix throughout a categorical representation $\cC$ of  $U_q(\fg)$, and let $w_0=s_{i_1}s_{i_2} \cdots s_{i_n}$ be a reduced expression.

\subsection{$\Theta_{w_0}$ on isotypic categorifications}	In this section we prove that $\Theta_{w_0}$ is t-exact on any isotypic categorification.  Fix $\lambda \in X_+$.  We write $\Theta=\Theta_{w_0}, L=L(\lambda)$ and $\cL=\cL(\lambda)$.

\begin{Lemma}\label{lem:techbit}
Let $k\in\{2,\hdots,n\}$ and set $\mu=s_{i_k}\cdots s_{i_n}(w_0(\lambda))$.  The weight space $L(\lambda)_{\mu-\alpha_{i_{k-1}}}$ is zero.
\end{Lemma}

\begin{proof}
By \cite[Proposition 21.3]{Humph}, it suffices to find $u\in W$ such that $u(\mu-\alpha_{i_{k-1}}) \nsucc w_0(\lambda)$.  Take $u=s_{i_{k-1}}$ and, noting that $(\mu,\alpha_{i_{k-1}})<0$, the result follows.
\end{proof}

Recall that $v_\lambda$ and $v_\lambda^{low}$ are the highest and lowest weight elements of the canonical basis $\bB(\lambda)$.

\begin{Lemma}\cite[Comment 5.10]{KT}\label{lemma:KT}
We have $\st_{w_0}(v_\lambda^{low})=v_\lambda$.
\end{Lemma}

\begin{Proposition}\label{prop:lowestwt}
The equivalence $\Theta\bone_{w_0(\lambda)}:D^b(\cL_{w_0(\lambda)})\to D^b(\cL_{\lambda})$ satisfies $\kk_{low}\mapsto\kk_{high}$, where both are considered as complexes concentrated in degree zero.  In particular, under the equivalences $\cL_\lambda \cong \cL_{w_0(\lambda)}\cong \kk\mmod_\ZZ$,  $\Theta\bone_{w_0(\lambda)}$ is isomorphic to the identity autofunctor of $D^b(\kk\mmod_\ZZ)$.
\end{Proposition}

\begin{proof}
Consider first the case $\fg=\fsl_2$.  On the minimal categorification of highest weight $m$, we have that $\Theta\bone_{-m}(\kk_{low})=\kk_{high}\langle \ell \rangle$ for some $\ell$ by \cite[Theorem 6.6]{CR}.  Since $[\Theta\bone_{-m}]=\st_{1}\sone_{-m}$, by Lemma \ref{lemma:KT} we conclude that $\ell=0$, and hence $\Theta\bone_{-m}(\kk_{low})=\kk_{high}$.

For general $\fg$, suppose $X \in \cL_\nu$ is simple and  $\sF_iX=0$ for some $i \in I$.  Consider $\cL$ as a categorical representation of $\fsl_2$ by restriction to the $i$-th root subalgebra. 
Then for some $m$ we have a morphism of categorical $\fsl_2$ representations
$
R_X:\cL(m) \to \cL,
$
such that $R_X(\kk_{low})=X$ \cite[Theorem 5.24]{CR}.  

The functor $R_X$ is equivariant for the categorical $\fsl_2$ action on $\cC$ determined by $\sE_i,\sF_i$ (in fact it is strongly equivariant in the sense of \cite[Definition 3.1]{LoWe}), and hence commutes with $\Theta_i\bone_\nu$.  Therefore we have that 
\begin{align*}
\Theta_i\bone_{\nu}(X) &\cong  \Theta_i\bone_{\nu}(R_X(\kk_{low})) \\
&\cong R_X(\Theta_i\bone_{\nu}(\kk_{low})) \\
&\cong R_X(\kk_{high}) \in \cL,
\end{align*}
and so $\Theta_i\bone_{\nu}(X)$ is in homological degree zero.  
It follows that in the case when $X$ is not necessarily simple (but still assume that $\sF_iX=0$), $\Theta_i\bone_{\nu}(X)$ is still in homological degree zero.  Indeed, by induction on the length of a Jordan-H\"older filtration of $X$ one deduces this since $\cL \subset D^b(\cL)$ is extension closed. 

%
%
%
Now we study $\Theta\bone_{w_0(\lambda)}$ applied to $\kk_{low}$.  For $k=2,\hdots,n$, by Lemma \ref{lem:techbit}, $$\sF_{i_{k-1}}(\Theta_{i_k}\cdots\Theta_{i_n}\bone_{w_0(\lambda)}(\kk_{low}))=0.$$ By the previous paragraph, it follows that $\Theta_{i_{k-1}}\cdots\Theta_{i_n}\bone_{w_0(\lambda)}(\kk_{low})$ is in homological degree zero, and in particular, $\Theta\bone_{w_0(\lambda)}(\kk_{low})$ is supported in homological degree zero.  Since in addition $$[\Theta\bone_{w_0(\lambda)}]=\st_{w_0}\sone_{w_0(\lambda)},$$ by Lemma \ref{lemma:KT} we conclude that $\Theta\bone_{w_0(\lambda)}(\kk_{low})\cong\kk_{high}$.
\end{proof}

\begin{Theorem}\label{thm:mincat}
Let $\lambda \in X_+$ and set $\cL=\cL(\lambda)$.  For any $\mu \in X$, $$\Theta\bone_\mu[-n]:D^b(\cL_\mu) \to D^b(\cL_{w_0(\mu)})$$ is t-exact, where $n=ht(\mu-w_0(\lambda))$.
\end{Theorem}

\begin{proof}
Consider  $P=\sE_{i_1}\cdots \sE_{i_\ell}(\kk_{low}) \in \cL_\mu$.  We will first prove by induction on $n$ that there exists an integer $k$ such that 
\begin{align}\label{eq:ind}
\Theta(P)[-n] \cong \sF_{j_1}\cdots \sF_{j_\ell}(\kk_{high})\langle k \rangle \in \cL_{w_0(\mu)},
\end{align}
where $j_r=\tau(i_{r})$.

The base case when $n=0$ follows by Proposition \ref{prop:lowestwt}.  For the inductive step write $P=\sE_{i_1}(Q)$. Note that $Q\in \cL_{\mu-\alpha_{i_1}}$.  By Proposition \ref{prop:mainrel} we have that 
\begin{align*}
\Theta\bone_\mu(P)[-n] &= \Theta \sE_{i_1}\bone_{\mu-\alpha_{i_1}}(Q)[-n] \\
&\cong \sF_{\tau(i_1)}\Theta\bone_{\mu-\alpha_{i_1}}(Q)[-n+1]\langle (\mu-\alpha_{i_1},\alpha_{i_1}) \rangle.
\end{align*}
By hypothesis $$\Theta\bone_{\mu-\alpha_{i_1}}(Q)[-n+1] \cong \sF_{j_2}\cdots \sF_{j_\ell}(\kk_{high})\langle k \rangle \in \cL_{w_0(\mu-\alpha_{i_1})}$$ for some $k$, and hence Equation (\ref{eq:ind}) follows.

Since up to grading shift, any projective indecomposable object in $\cL_\mu$, respectively $\cL_{w_0(\mu)}$, is a summand of an object of the form $\sE_{i_1}\cdots \sE_{i_\ell}(\kk_{low})$ (respectively $\sF_{j_1}\cdots \sF_{j_\ell}(\kk_{high})$), it follows that $\Theta\bone_\mu[-n]$ takes projective objects in $\cL_\mu$ to projective objects in $\cL_{w_0(\mu)}$.  Since $\Theta\bone_\mu[-n]$ is a derived equivalence it follows that it is t-exact.

\end{proof}

\begin{Remark}
Theorem \ref{thm:mincat} is a generalisation of  \cite[Theorem 6.6]{CR}, which covers the $\fsl_2$ case.  Note that \cite[Theorem 6.6]{CR} is crucial in the work of Chuang and Rouquier, since it's one of the main technical results needed to prove that Rickard complexes are invertible.  Our proof in the general case follows a completely different approach, but   it does not give a new proof in the case of $\fsl_2$. Indeed we use \cite[Theorem 6.6]{CR} explicitly in the proof of Proposition \ref{prop:lowestwt}, and more generally we use the fact the $\Theta_i$ is invertible throughout.  
\end{Remark}

\begin{Corollary}\label{cor:mainthm}
Suppose $\cC$ is an isotypic categorification of type $\lambda$, for some $\lambda \in X_+$, and let $\mu \in X$.  Then $\Theta\bone_\mu[-n]:D^b(\cC_\mu) \to D^b(\cC_{w_0(\mu)})$ is a t-exact equivalence, where $n=ht(\mu-w_0(\lambda))$.  
\end{Corollary}

\begin{proof}
By Lemma \ref{lem:isocat}, there exists an abelian $\kk$-linear category $\cA$ such that $\cC \cong \cL(\lambda)\otimes_\kk \cA$ as categorical representations.  We have that 
$$
\Theta\bone_\mu[-n](\cL(\lambda)_\mu\otimes_\kk \cA) \cong \Theta\bone_\mu[-n](\cL(\lambda)_\mu) \otimes_\kk \cA \cong \cL(\lambda)_{w_0(\mu)}\otimes_\kk \cA,
$$
proving that $\Theta\bone_\mu[-n]:D^b(\cC_\mu) \to D^b(\cC_{w_0(\mu)})$ is a t-exact equivalence.
\end{proof}

\subsection{$\Theta_{w_0}$ on general categorical representations}\label{sec:mainpf}
In this section we prove that $\Theta_{w_0}$ is a perverse equivalence on an arbitrary categorical representation.
Fix $\mu \in X$ such that $\cC_\mu$ is nonzero.  For ease of notation, set $\cA=\cC_\mu$ and $\cA'=\cC_{w_0(\mu)}$.

Consider a filtration by Serre subcategories
\begin{align*}
0=\cC_0\subset \cC_1\subset \cdots \subset \cC_r=\cC,
\end{align*}
which can be either the Jordan-H\"older filtration (Theorem \ref{thm:JH}) or the isotypic filtration (Remark \ref{rem:isofilts}).  So for every $i$,  $\cC_i$ is a subrepresentation of $\cC$, and $\cC_i/\cC_{i-1}$ is either a simple categorification or an isotypic one.  Define $\lambda_i \in X_+$ by requiring that $[\cC_i/\cC_{i-1}]_{\CC(q)}$ is a representation of type $\lambda_i$.

Construct filtrations of $\cA$ and $\cA'$ by $\cA_i=\cC_i\cap \cA, \cA_i'=\cC_i \cap \cA'$.  These are Serre subcategories of $\cA$ and $\cA'$ respectively.  Let $p:\{0,...,r\}\to \ZZ$ be given by 
$p(i)=ht(\mu-w_0(\lambda_i))$.  

\begin{Theorem}
\label{thm:perveq}
$\Theta_{w_0}\bone_\mu:D^b(\cA) \to D^b(\cA')$ is a perverse equivalence with respect to $(\cA_\bullet, \cA_\bullet',p)$ for either the Jordan-H\"older or the isotypic filtration. 
\end{Theorem}

%
\begin{proof}
Since $\cC_i \subset \cC$ is a categorical subrepresentation, the terms of the functor $\Theta_{w_0}\bone_\mu$ leave $\cC_i$ invariant, and in particular take objects in $\cA_i$ to $\cA_i'$.  By Lemma \ref{lem:trianginv} this implies that $\Theta_{w_0}\bone_\mu (D^b_{\cA_i}(\cA)) \subseteq D^b_{\cA_i'}(\cA')$.

Now, $\cC_i/\cC_{i-1}$ is a simple or isotypic categorification (of type $\lambda_i$).  By Corollary \ref{cor:mainthm}, $\Theta_{w_0}\bone_\mu[-p(i)]$ restricts to an abelian equivalence $\cA_i/\cA_{i-1} \to \cA_i'/\cA_{i-1}'$, i.e. the functor 
$$
\Theta_{w_0}\bone_\mu[-p(i)]:D^b_{\cA_i}(\cA)/D^b_{\cA_{i-1}}(\cA) \to D^b_{\cA'_i}(\cA')/D^b_{\cA'_{i-1}}(\cA')
$$
is a t-exact equivalence.  This shows that $\Theta_{w_0}\bone_\mu$ is a perverse equivalence with respect to $(\cA_\bullet, \cA'_\bullet, p)$.
\end{proof}

\begin{Remark}
The $\fsl_2$ case of Theorem \ref{thm:perveq} appears as \cite[Proposition 8.4]{CRperv}, by a different argument relying on a technical lemma \cite[Lemma 4.12]{CRperv}.  
\end{Remark}

\section{Crystalising the braid group action}\label{sec:cacti}

Already in the work of Chuang and Rouquier, a close connection is established between categorical representation theory and the theory of crystals (although it is not phrased in this language, cf. Proposition \ref{prop:cryst} below).  In this section we describe a new component of this theory.  More precisely, let $V$ be an integrable representation of $U_q(\fg)$.  Recall that Lusztig has defined a braid group action on $V$ \cite{Lusbook}.
In this section we explain how to use our results to ``crystalise'' this braid group action to obtain a cactus group action on the crystal of $V$, recovering the recently discovered action by generalised Sch\"utzenberger involutions. 

\subsection{Cactus groups}\label{sec:cactus}
The cactus group associated to the Dynkin diagram $I$ has several incarnations.  Geometrically, it appears  as the fundamental group of a  space associated to the Cartan subalgebra $\fh$ of $\fg$. Namely, let $\fh^{\text{reg}}\subseteq \fg$ denote the regular elements of $\fh$. The cactus group $C=C_I$ is the $W$-equivariant fundamental group of the real locus of the de Concini-Procesi wonderful compactification of $\fh^{\text{reg}}$ (see \cite{DJS03}, \cite[Section 2]{HKRW} for further details):
$$C=\pi^W_1(\overline{\PP(\fh^{\text{reg}})}(\RR)).$$
There is a surjective map $C \to W$, and the kernel of this map is called the pure cactus group.  In type A it is  the fundamental
group of the Deligne-Mumford compactification of the moduli space of real genus
$0$ curves with $n + 1$ marked points \cite{HK06}. 

The cactus group has a presentation using Dynkin diagram combinatorics. For any subdiagram $J \subseteq I$, recall that $\tau_J:J\to J$ is the diagram automorphism induced by the longest element $w_0^J \in W_J$.
\begin{Definition}
	The \textbf{cactus group} $C=C_I$ is generated by $c_J$, where $J \subseteq I$ is a connected subdiagram, subject to the following relations:
	\begin{itemize}
		\item[(i)] $c_J^2=1$ for all $J \subseteq I$,
		\item[(ii)] $c_Jc_K=c_Kc_J$, if $J \cap K=\emptyset$ and there are no edges connecting any $j \in J$ to any $k\in K$, and
		\item[(iii)] $c_Jc_K=c_Kc_{\tau_K(J)}$ if $J \subseteq K$.
	\end{itemize}
\end{Definition}
The surjective map $C \to W$ mentioned above is given by $c_J \mapsto w_0^J$. We are interested in the cactus group in connection to the theory of crystals. 

A $\fg$-crystal $\bB$ is called \textbf{normal} if it is isomorphic to a disjoint union $\sqcup_\la \bB(\la)$ for some collection of highest weights $\la$. The category of normal $\fg$-crystals has the structure of a coboundary category analogous to the braided tensor category structure on $U_q(\fg)$-representations. It is realized through an ``external'' cactus group action of $C_{A_{n-1}}$ on $n$-tensor products of $\fg$-crystals, described by Henriques and Kamnitzer \cite[Theorems 6,7]{HK06}. 

We are interested in the ``internal'' cactus group action of $C$ on any $\fg$-crystal $\bB$. Both the internal and external actions 
rely on the following combinatorially defined maps, which are generalisations of the partial Sch\"utzenberger involutions in type $A$.

\begin{Definition} \label{def:genSchutz}
The \textit{generalised Sch\"utzenberger involution} $\xi_\la$ on $\bB(\la)$ is the set map defined uniquely by the following properties. For all $b \in \bB(\la)$ and $i \in I$:
	\begin{enumerate}
		\item $\wt(\xi_\la(b))=w_0\wt(b)$,
		\item $\xi_\la \tilde{e}_i(b)=\tilde{f}_{\tau(i)} \xi_\la(b)$,
		\item $\xi_\la \tilde{f}_i(b)=\tilde{e}_{\tau(i)} \xi_\la(b)$.
	\end{enumerate}	
	Note that from (1), $\xi_\la$ swaps the lowest and highest weight element of $\bB(\la)$, and the rest of its behavior is then determined from (2) and (3). The generalised Sch\"utzenberger involution $\xi$ on $\bB=\sqcup_\la \bB(\la)$ is the set map which acts as $\xi_\la$ on each irreducible component $\bB(\la)$.
\end{Definition}

Note that (1) implies that $\xi_\la$ maps the lowest  weight element to the highest weight element (and vice-versa), and then (2) and (3) ensure that it is uniquely defined.  

For $J \subseteq I$, denote by $\bB_J$ the crystal $\bB$ restricted to the subdiagram $J$.  We denote the corresponding Sch\"utzenberger involution by $\xi_J$. 

\begin{Theorem}(\cite[Theorem 5.19]{HKRW})\label{thm:cactact}
	For any $\fg$-crystal $\bB$, the assignment $c_J \mapsto \xi_{J}$ defines a (set-theoretic) action of $C$ on $\bB$.
\end{Theorem}

\subsection{The cactus group action arising from Rickard complexes}
We now explain how  cactus group actions arise from categorical representations, analogous to the construction of the crystal on $\Irr(\cC)$ in Proposition \ref{prop:cryst}.

Let $\cC$ be a categorical representation of $U_q(\fg)$.  For any weight $\mu \in X$, by Theorem \ref{thm:perveq} $\Theta_{w_0}\bone_\mu:D^b(\cC_\mu) \to D^b(\cC_{w_0(\mu)})$ is a perverse equivalence, and hence it induces a bijection $\varphi_I1_\mu:\Irr(\cC_\mu) \to \Irr(\cC_{w_0(\mu)})$.  By varying $\mu$  we obtain a bijection 
$
\varphi_I:\Irr(\cC) \to \Irr(\cC).
$  

Now let $J\subseteq I$ be a connected subdiagram, and let $\fg_J \subset \fg$ be the corresponding subalgebra.  By restriction, $\cC$ is also a categorical representation of $U_q(\fg_J)$, and hence by the above discussion we also obtain a bijection $
\varphi_J:\Irr(\cC) \to \Irr(\cC).
$

We will prove that this family of bijections defines an action of the cactus group $\Irr(\cC)$.  First we need the following technical result.  The important point here is just that there exists an integer $n$ such that 
$\st^2_{w_0}\sone_\mu=\pm q^n\sone_\mu$.

\begin{Lemma}\label{lem:fulltwist}
Let $\lambda\in X_+, \mu \in X$, and let $\mu-w_0(\lambda)=\sum_{r=1}^\ell \alpha_{i_r}$, where $i_r\in I$.  Set $j_r=\tau(i_r)$ and 
define $n(\la,\mu)\in\ZZ$ by 
    $$
    n(\la,\mu)=2\left(\sum_{r=1}^\ell \la_{j_r}+1-\sum_{1\leq r\leq s\leq \ell}a_{j_rj_s}+(\la,\rho)\right)
    $$
    Then on  $L(\lambda)_\mu$ we have \begin{align}\label{eq:fulltwist}\st^2_{w_0}\sone_\mu=(-1)^{\langle 2\lambda ,\rho^\vee \rangle}q^{n(\la,\mu)}\sone_\mu.\end{align}
\end{Lemma}

\begin{proof}
We will prove the  claim by induction on $ht(\mu-w_0(\la))$.  For $\mu=w_0(\la)$, we have that $\ell=0$ so  $n(\la,\mu)=(\la,\rho)$.  By \cite[Equation (7)]{KT}  $\st_{w_0}(v_\la)=(-1)^{\langle \lambda ,\rho^\vee \rangle}q^{(\la,\rho)}v_\la^{low}$, which, combined with Lemma \ref{lemma:KT}, implies that $\st_{w_0}^2(v_\la^{low})=(-1)^{\langle \lambda ,\rho^\vee \rangle}q^{(\la,\rho)}v_\la^{low}$.  Since $\dim(L(\la)_{w_0(\la)})=1$, this proves the base case.  

Now choose any $\mu$ and suppose (\ref{eq:fulltwist}) holds for any weight $\mu'$ such that $ht(\mu'-w_0(\la))<ht(\mu-w_0(\la))$.  Consider  $v=E_{j_1}\cdots  E_{j_\ell}v_\la^{low} \in L(\la)_\mu$.  First note that by \cite[Lemma 5.4]{KT} we have that
\begin{align}\label{lem:ftcomm}
\st_{w_0}^2E_i=q^2K_i^{-2}E_i\st_{w_0}^2.
\end{align}
Setting $v'=E_{j_2}\cdots  E_{j_\ell}v_\la^{low}$, by induction we have
\begin{align*}
\st_{w_0}^2v &= q^2K_{j_1}^{-2}E_{j_1}\st_{w_0}^2v' \\
&= (q^2K_{j_1}^{-2}E_{j_1})(-1)^{\langle \lambda ,\rho^\vee \rangle}q^{n(\la,\mu-\alpha_{j_1})}v' \\
&= (-1)^{\langle \lambda ,\rho^\vee \rangle}q^{2+n(\la,\mu-\alpha_{j_1})-2(\mu,\alpha_{j_1})}v
\end{align*}
One checks easily that $n(\la,\mu)=2+n(\la,\mu-\alpha_{j_1})-2(\mu,\alpha_{j_1})$, proving that $\st_{w_0}^2v=(-1)^{\langle \lambda ,\rho^\vee \rangle}q^{n(\la,\mu)}v$.  Since this holds for any vector of the form $E_{j_1}\cdots  E_{j_\ell}v_\la^{low}$ in $L(\la)_\mu$, this completes the inductive step.
\end{proof}

\begin{Theorem}\label{thm:cactus}
The assignment $c_J \mapsto \varphi_J$ defines an action of $C$ on $\Irr(\cC)$.
\end{Theorem}

\begin{proof} We need to show that the bijections $\varphi_J$ satisfy the cactus group relations. 
%
%
%
%

\textit{Relation (i)}: Without loss of generality we may assume $J=I$.   Fix a weight $\mu$.  Our aim is to show that 
\begin{equation}
\label{eq:rel1}
\varphi_I \varphi_I1_\mu=\Id_{\Irr(\cC_\mu)}.
\end{equation}

Since the filtration of $\cC_{w_0(\mu)}$ which we use in the perversity data of $\Theta_{w_0}\bone_\mu$, agrees with the filtration of $\cC_{w_0(\mu)}$ which we use in the perversity data of $\Theta_{w_0}\bone_{w_0(\mu)}$, by Lemma \ref{lem:pervcomp} the composition $\Theta_{w_0} \Theta_{w_0}\bone_\mu$ is a perverse autoequivalence of $D^b(\cC_\mu)$. 

The  functor $[\langle 2\lambda ,\rho^\vee \rangle]\langle n(\lambda,\mu)\rangle$ is also a perverse autoequivalence of $D^b(\cC_\mu)$.   By Lemma \ref{lem:fulltwist} these two perverse equivalences induce the same map on Grothendieck groups, and hence by Lemma \ref{lem:LosLem} they also induce the same bijection.  Since the bijection induced by $[\langle 2\lambda ,\rho^\vee \rangle]\langle n(\lambda,\mu)\rangle$ is the identity, this proves  relation (i).



\textit{Relation (ii)}:  Let $J,K \subset I$ be disjoint subdiagrams with no connecting edges.  Our aim is to show that
\begin{align}\label{eq:rel2}
\varphi_J\varphi_K1_\mu = \varphi_K \varphi_J1_\mu.
\end{align}
We prove a slightly more general statement, namely that for any categorical representation $\cC$ of $U_q(\fg_J \times\fg_K)$, relation (\ref{eq:rel2}) holds.

Note that $\Theta_J\Theta_K\bone_\mu\cong \Theta_K\Theta_J\bone_\mu$ are isomorphic perverse equivalences, so they induce the same bijections.  It remains to show that
\begin{align}\label{eq:rel2again}
\varphi_{\Theta_J\Theta_K\bone_\mu} =\varphi_J\circ\varphi_K 1_\mu.
\end{align}

Consider first the case when $\cC$ categorifies an simple representation of $U_q(\fg_J\times\fg_K)$.  A minimal categorification of $U_q(\fg_J\times\fg_K)$ is of the form $\cL(\lambda)\otimes_\kk\cL(\mu)$, where $\lambda$ is a highest weight for $\fg_J$ and $\mu$ is a highest weight for $\fg_K$.  Hence by Lemma \ref{lem:isocat}, a simple categorification of $U_q(\fg_J\times\fg_K)$ is of the form $\cL(\lambda)\otimes_\kk\cL(\mu)\otimes_\kk\cA$ for some abelian category $\cA$.  

This implies that as a categorical representation of $U_q(\fg_J)$ (respectively $U_q(\fg_K)$), $\cC$  categorifies an isotypic representation.  
By Corollary \ref{cor:mainthm} $\Theta_J\bone_{w_0^K(\mu)}$ and $\Theta_K\bone_\mu$ are t-exact up shift on isotypics categorifications.  Hence Equation (\ref{eq:rel2again}) follows by Lemma \ref{lem:pervcomp}.

Now consider a Jordan-H\"older filtration (Theorem \ref{thm:JH}):
\begin{align*}
0=\cC_0\subset \cdots\subset \cC_n=\cC,
\end{align*}
where for every $i$, $\cC_i$ is a  subrepresentation of $\cC$, and $\cC_i/\cC_{i-1}$ is a simple categorification of $U_q(\fg_J\times\fg_K)$.  Equation (\ref{eq:rel2again}) now follows by an easy induction on $i$.
Indeed the base case when $i=1$ holds by the paragraph above, and the inductive step by Lemma \ref{lem:quot}.


\textit{Relation (iii)}: We need to show that $\varphi_J\varphi_K1_\mu=\varphi_K\varphi_{\tau_K(J)}1_\mu$, where $J\subset K$.  Again, we may assume that $K=I$.  Note that we have an isomorphism at the level of functors:
\begin{align*}
\Theta_{w_0}^{-1}\Theta_{w_0^J}\Theta_{w_0}\bone_\mu \cong \Theta_{w_0^{\tau_K(J)}}\bone_\mu,
\end{align*}
which lifts the corresponding relation in $B$.  Since this is an isomorphism of perverse equivalences, they must induce the same bijections by Lemma \ref{lem:LosLem}.  

It remains to show that 
\begin{align}\label{eq:compo}
\varphi_{\Theta_{w_0}^{-1}\Theta_{w_0^J}\Theta_{w_0}\bone_\mu} =\varphi_I^{-1}\circ \varphi_J\circ \varphi_I.
\end{align}
When $\cC$ is a simple categorification, by Corollary \ref{cor:mainthm} $\Theta_{w_0}\bone_\mu$ is t-exact (up to shift).  Hence Equation (\ref{eq:compo}) follows by Lemma \ref{lem:texactperv}.  Now apply the same reasoning as in the proof of Relation (ii) to deduce equation (\ref{eq:compo}) in the general case.

%

\end{proof}

\subsection{Reconciling the two cactus group actions}

Let $\cC$ be a categorical representation of $U_q(\fg)$, and consider the $\fg$-crystal $\bB=\Irr(\cC)$.   There are two actions of the cactus group on $\bB$, the first arising combinatorially via  Sch\"utzenberger involutions (Theorem \ref{thm:cactact}) and the other categorically via Theorem \ref{thm:cactus}. 

\begin{Theorem}\label{thm:agree}
The two actions of the cactus group on $\bB$  agree.
\end{Theorem}

\begin{proof}
It suffices to show that $\varphi_I=\xi_I$.  
First, suppose $\cC$ is a simple categorification of type $\lambda \in X_+$.  In this case $\bB=\bB(\lambda)$, and $\xi=\xi_I$ is determined by:
\begin{align*}
\xi(v_\lambda)&=v_\lambda^{low}, \text{ and } \\
\xi(\tilde{e}_i(v))&=\tilde{f}_{\tau(i)}(\xi(v)) \text{ for all } v \in \bB,
\end{align*}
so we need to show that $\varphi_I$ satisfies these properties as well.  

The first is an immediate consequence of Corollary \ref{cor:mainthm}.  To show that $\varphi_I$ satisfies the second property, fix $\mu \in X$ and $ i \in I$.  We set $n=ht(\mu-w_0(\lambda)), j=\tau(i)$, and write $\Theta=\Theta_{w_0}$.  

Consider the following diagram:
\begin{equation}
\begin{tikzcd}
D^b(\cC_\mu) \arrow[rrr,"\Theta\bone_\mu\sh{-n}\langle \mu_i \rangle"] \arrow[d,"\sE_i\bone_\mu"] &&&  D^b(\cC_{w_0(\mu)}) \arrow[d,"\sF_j\bone_{w_0(\mu)}"] \\
D^b(\cC_{\mu+\alpha_i}) \arrow[rrr,"\Theta\bone_{\mu+\alpha_i}\sh{-n-1}"]  &&&  D^b(\cC_{w_0(\mu)-\alpha_j})
\end{tikzcd}
\end{equation}
By  Proposition \ref{prop:mainrel} this diagram commutes (note that we shifted both sides of the equation by $-n-1$).  By Theorem \ref{thm:mincat} both horizontal arrows are in fact t-exact equivalences so this restricts to a diagram of abelian categories:
\begin{equation}
\begin{tikzcd}
\cC_\mu \arrow[rrr,"\Theta\bone_\mu\sh{-n}\langle \mu_i \rangle"] \arrow[d,"\sE_i\bone_\mu"] &&&  \cC_{w_0(\mu)} \arrow[d,"\sF_j\bone_{w_0(\mu)}"] \\
\cC_{\mu+\alpha_i} \arrow[rrr,"\Theta\bone_{\mu+\alpha_i}\sh{-n-1}"]  &&&  \cC_{w_0(\mu)-\alpha_j}
\end{tikzcd}
\end{equation}

Let $L\in \cC_\mu$ be a simple object,  and let $L'=\Theta\bone_\mu(L)\sh{-n}\langle \mu_i \rangle$.  Note that $L'\in \cC_{w_0(\mu)}$ is simple and $\varphi_I(L)=L'$.  By the above diagram we have an isomorphism
\begin{align*}
\Theta\bone_{\mu+\alpha_i}(\sE_i(L))\sh{-n-1} \cong \sF_j(L').
\end{align*}
Now, $\widetilde{\sF}_j(L') \subset \sF_j(L')$ is the unique simple subobject.  On the other hand, since $\Theta\bone_{\mu+\alpha_i}\sh{-n-1}$ is an abelian equivalence, $\Theta\bone_{\mu+\alpha_i}(\widetilde{\sE}_i(L))\sh{-n-1}\subset \Theta\bone_{\mu+\alpha_i}(\sE_i(L))\sh{-n-1}$ is a simple subobject.  Therefore
\begin{align*}
\Theta\bone_{\mu+\alpha_i}(\widetilde{\sE}_i(L))\sh{-n-1} \cong \widetilde{\sF}_j(L').
\end{align*}
Since the equivalence class of the left hand side is $\varphi_I\circ \tilde{e}_i(L)$, this shows that $\varphi_I$ satisfies the second defining property, and hence the two cactus group actions agree in the case of a simple categorification. 

The general case when $\cC$ is not necessarily a simple categorification follows easily using the Jordan-H\"older filtration (Theorem \ref{thm:JH}) and Lemma \ref{lem:quot}.
\end{proof}

\section{Examples and Applications}
\subsection{Examples}
We now examine three examples.  The first two consider minimal categorifications of the adjoint representation, while the third studies the categorification of the $n$-fold tensor product of the standard representation of $\fsl_n$.  For ease of presentation, we ignore gradings and consider non-quantum categorical representations.

\begin{example}
Let's consider the first non-trivial example of Theorem \ref{thm:mincat}: the minimal categorification of the adjoint representation of $\fsl_2$.  We can model this as follows:
$$
\begin{tikzcd}
\kk\mmod \arrow[rr, bend left, "ind"] &  & R\mmod \arrow[rr, bend left,"res"] \arrow[ll, bend left, "res"] &  & \kk\mmod \arrow[ll, bend left, "ind"]
\end{tikzcd}
$$
where $R=\kk[x]/(x^2)$ \cite[Example 5.17]{CR}.  Here $\kk\mmod$ is the $\pm2$ weight category, and $R\mmod$ is the zero weight category.  The arrows describe the $\sE,\sF$ functors (we omit the higher structure). 

Consider the Rickard complex $\Theta=\Theta\bone_0:D^b(R\mmod) \to D^b(R\mmod)$.  For $M\in R\mmod$, we have:
\begin{align*}
\Theta(M)=R\otimes_\kk M \to M,
\end{align*}
where the differential is given by the action map, and $M$ is in  cohomological degree $0$.  It's an exercise to verify that $\Theta(M)$ is quasi-isomorphic to $M'[1]$, where $M'$ is the twist of $M$ by the automorphism of $R$ given by $a+bx \mapsto a-bx$.  This shows that $\Theta[-1]$ is the t-exact equivalence $M\mapsto M'$.
\end{example}

\begin{example}
More generally, one can consider the minimal categorification of the adjoint representation of a simple simply-laced Lie algebra $\fg$.  This was studied by Khovanov and Huerfano in \cite{HuerKh}, who used zigzag algebras to model this category.  

For a weight $\alpha$ of the adjoint representation of $\fg$, the weight  category $\cC_\alpha$ is taken to be $\kk\mmod$ as long as $\alpha\neq 0$.  However, the zero weight  category is more interesting:  $\cC_0 := A \mmod$, where $A$ is the zigzag algebra associated to the Dynkin diagram $I$ of $\fg$ (cf. \cite{HuerKh} for the precise definition).  

The isomorphism classes of indecomposable projective left $A$-modules $\{P_i\}$ and the isomorphism classes of indecomposable projective right $A$-modules $\{Q_{i}\}$ are both indexed by $i \in I$.  Tensoring with these modules defines functors 
\begin{align*}
	&\sE_i: \cC_0 \longrightarrow \cC_{\alpha_i}, \quad M \mapsto Q_{i} \otimes_{A} M, \\
	&\sE_i : \cC_{-\alpha_i} \longrightarrow \cC_0, \quad  V \mapsto  P_i\otimes_\kk V.
\end{align*}
The functors $\sF_i$ are defined analogously and are biadjoint to the $\sE_i$.

The Rickard complex
$\Theta_i \bone_0$ is given by tensoring with the complex
$$
	P_i\otimes_\kk Q_i \rightarrow A
$$
of $(A,A)$-bimodules, where the differential is given by the multiplication in $A$, and $A$ sits in cohomological degree zero.  These functors are autoequivalences of the derived category $D^b(\cC_0)$.  

By Theorem \ref{thm:mincat}  $\Theta_{w_0}\bone_0[-n]$ is t-exact, where $n+1$ is the Coxeter number of $\fg$.  
This auto-equivalence can be explicitly described as follows: the automorphism of the Dynkin diagram $\tau:I\to I$ induces an automorphism $\psi$ of $A$.  Then we claim that $\Theta_{w_0}\bone_0[-n]$ is the abelian autoequivalence of $A\mmod$ defined by twisting with $\psi$.  Indeed, since $\Theta_{w_0}\bone_0[-n]$ is an abelian autoequivalence, it is determined up to isomorphism by its action on simple objects.  Moreover, the action on simple objects can be read off from the action of $W$ on the Grothendieck group of $A \mmod$ as follows: $[A \mmod]_\ZZ$ is isomorphic to the Cartain subalgebra by mapping $[L_i]$ ($L_i$ is the simple head of $P_i$) to the simple root vector $H_i$.  And on the root vectors we have that $w_0\cdot H_i = -H_{\tau(i)}$ 
\end{example}

\begin{example}\label{ex:catO}
Let $\fg=\fsl_n$ and consider the $n$-fold tensor power of the standard representation $V^{\otimes n}$.  Categorifications of $V^{\otimes n}$ have been well-studied, and a model $\cC$ for this categorical representation can be constructed using the BGG category $\cO$ of $\fg$ \cite{Sussan2007, MS09}.  In this model, the principal block $\cO_0 \subset \cO$ appears as the zero weight category of $\cC$, and the Rickard complexes acting on $D^b(\cO_0)$ are the well-known shuffling functors.

By Theorem \ref{thm:perveq} $\Theta_{w_0}\bone_0:D^b(\cO_0)\to D^b(\cO_0)$ is a perverse equivalence with respect to an isotypic filtration.  In fact, this recovers the type A case of a theorem of the third named author \cite{LosCacti}, using completely different methods (in \cite{LosCacti} the perversity of $\Theta_{w_0}$ is proved using the theory of W-algebras).  We can interpret the filtration of $\cO_0$ arising from our perspective concretely using the Robinson-Schensted correspondence .  

Recall that the simple objects in $\cO_0$ are the irreducible highest weight representations $L(w), w \in S_n$, where $L(w)$ has highest weight $w\rho-\rho$ ($\rho$ is the half-sum of positive roots of $\fsl_n$). 

We view a partition $\lambda$ of $n$  simultaneously as a dominant integral weight for $\fsl_n$, and as an index for the irreducible Specht module $S^\lambda$ of the symmetric group $S_n$. Let $\SYT(\lambda)$ denote the set of standard Young tableau of shape $\lambda$, and let $d(\lambda)=|\SYT(\lambda)|$.  Recall that $d(\lambda)=\dim S^\lambda$.  

Choose an ordering of the partitions of $n$, $\lambda_1,\hdots,\lambda_r$, so that if $\lambda_i \prec \lambda_j$ in the dominance order, then $i<j$.  Note that the dominance order on partitions of $n$ is equivalent to the positive root ordering on partitions (thought of as weights for $\fsl_n$).
By Remark \ref{rem:isofilts} there is an isotypic filtration on $\cC$,
$
0 \subset \cC_1 \subset \cdots\subset \cC_N=\cC,
$
where  $\cC_i/\cC_{i-1}$ is an isotypic categorification of type $\lambda_i$.  Then  $\Theta_{w_0}$ is a perverse equivalence with respect to the filtration $\cO_{0,i} = \cO_0\cap \cC_i$ and the perversity function $p(i)=ht(\lambda_i)$.

We would like now to define the categories $\cO_{0,i}$ more explicitly using the Robinson-Schensted correspondence, which recall is a bijection \cite{Sagan}:
\begin{align*}
\mathsf{RS}:S_n \longrightarrow \bigsqcup_{j=1}^r\SYT(\lambda_j)\times\SYT(\lambda_j), \; w \mapsto (P(w),Q(w)).
\end{align*}
We will use a crystal analogue of classical Schur-Weyl duality.  This is given by a crystal isomorphism: 
\begin{align}\label{eq:crystalSW}
\bB(\varpi_1)^{\otimes n} \longrightarrow \bigsqcup_{j=1}^r \bB(\lambda_j)\times\SYT(\lambda_j).
\end{align}

Now, recall that for a partition $\la$ of $n$, the underlying set of the crystal $\bB(\la)$ can be chosen to be the set of semistandard Young tableaux of shape $\la$ with entries $1,\hdots,n$, and the weight zero subset of $\bB(\la)$ is precisely $\SYT(\la)$.  
The essential point is that the isomorphism \eqref{eq:crystalSW} can be chosen so that it restricts to the map $\mathsf{RS}$ on the elements of weight zero (\cite[Theorem 3.5]{Shimozono}). (Note that  the elements of weight zero in $\bB(\varpi_1)^{\otimes n}$ are naturally identified with permutations of $1,\hdots,n$.)

This shows that as an element of the crystal $\bB(\varpi_1)^{\otimes n}$, $[L(w)]$ is in a connected component whose highest weight is the shape of $Q(w)$ (or equivalently $P(w)$).   
Therefore, following Remark \ref{rem:JHfromCrystal}, we can construct the isotypic filtration be defining  $\cO_{0,i}$ to be  the Serre subcategory of $\cO_0$ generated by $L(w)$ such that the shape of $Q(w)$ is among $\la_1,\hdots,\la_i$.

\end{example}

\subsection{Type A combinatorics}\label{sec:typeA}

In this final section, we specialise to type A and discuss  the combinatorics of Kazhdan-Lusztig bases and standard Young tableaux from the vantage of perverse equivalences.   

Set  $I=A_{n-1}=\{1,\hdots,n-1\}$.
We continue with the notation in Example \ref{ex:catO} and view a partition $\lambda \vdash n$ simultaneously as a dominant integral highest weight for $\fsl_n$, and as an index of the Specht module $S^\lambda$. Recall that by Schur-Weyl duality, $L(\lambda)_0$ is isomorphic to $S^\lambda$.  The Kazhdan-Lusztig basis of the Hecke algebra naturally descends to a basis of $S^\la$, which we denote $\{C_T\}$ indexed by   $T\in\SYT(\lambda)$.  For further details we recommend the exposition in \cite{Rhoad}.  

Consider the  minimal categorification $\cL(\lambda)$, and in particular its zero weight category $\cL(\lambda)_0$.  For convenience, we forget the grading and work in the non-quantum setting.  The simple objects $L(T)  \in \cL(\lambda)_0$ are indexed by $T\in \SYT(\lambda)$, and hence a perverse equivalence $\sF:D^b( \cL(\lambda)_0) \to D^b( \cL(\lambda)_0)$ induces a bijection $\varphi_\sF:\SYT(\lambda) \to \SYT(\lambda)$.  

The bijection $\varphi_I$ studied in the previous section specialises to the well-known Sch\"utzenberger involution on standard Young tableau, otherwise known as the ``evacuation operator'' $e$ \cite{Sagan}.  Indeed, by Theorem \ref{thm:agree} $\varphi_I$ recovers the cactus group action on the crystal $\Irr(\cC)$ by generalised Sch\"utzenberger involutions.  The Sch\"utzenberger involution is well-known to agree with the evacuation operator \cite[Theorem A1.2.10]{StanleyII}).  Note that this is elementary: it follows directly from the fact that the evacuation operator satisfies the properties of Definition \ref{def:genSchutz} on standard Young tableaux.

The promotion operator $j:\SYT(\lambda)\to \SYT(\lambda)$ is another important function in algebraic combinatorics, which is closely connected to the RSK correspondence and related ideas such as jeu de taquin.  
Letting $J=\{1,\hdots,n-2\}\subset I$, we can express promotion in terms of the Sch\"utzenberger involution: $j=\varphi_I\varphi_{J}$.
We refer the reader to \cite{Sagan} for a detailed exposition.  We can now see easily that promotion also arises from a perverse equivalence:

\begin{Proposition}\label{prop:promo}
Let $c_n=(1,2,\hdots,n)\in S_n$ be the long cycle.  Then $$\Theta_{c_n}:D^b( \cL(\lambda)_0) \to D^b( \cL(\lambda)_0)$$ is a perverse equivalence whose associated bijection is the promotion operator: $\varphi_{\Theta_{c_n}}=j$.
\end{Proposition}

\begin{proof}
Notice that $c_n=w_0w_0^J$ for $J$ as above. 
Now recall that $\Theta_{w_0}$ is (up to shift) a t-exact  autoequivalence of $D^b( \cL(\lambda)_0)$ (Theorem \ref{thm:mincat}).  Since $\Theta_{w_0^J}$ is a perverse equivalence (Theorem \ref{thm:perveq}),  its inverse is too.  Therefore $\Theta_{c_n}\cong\Theta_{w_0}\Theta_{w_0^J}^{-1}$ is also a perverse autoequivalence. 
By Lemma \ref{lem:texactperv} we have that $\varphi_{\Theta_{c_n}}=\varphi_I\varphi_J$, and hence we recover the promotion operator.
\end{proof}

We can also use this set-up to extract information about the action of $S_n$ on the Kazhdan-Lusztig basis of $S(\lambda)$.  This is based on the following elementary lemma:

\begin{Lemma}\label{lem:KLact}
Let $w \in S_n, \lambda \vdash n$ and suppose $\Theta_w :D^b(\cL(\lambda)_0) \to D^b(\cL(\lambda)_0)$ is t-exact up to shift. 
Then for any $T \in \SYT(\lambda)$, $w\cdot C_T=\pm C_{S}$, where $S=\varphi_{\Theta_w}(T)$.
\end{Lemma}

\begin{proof}
Since $\Theta_w$ is t-exact up to shift, we have: 
\begin{align*}
[\Theta_{w}(L(T))]=\pm [L(S)].
\end{align*}
The result now follows since the isomorphism $[\cL(\lambda)_0]_\CC \cong S^\lambda, L(T) \mapsto C_T$, is $S_n$-equivariant, and  by \cite[Proposition 10]{BauqWeyl}, the action of the braid group $B=B_n$ on $L(\lambda)_0$  factors through $S_n$.
\end{proof}

Applying this lemma to Theorem \ref{thm:mincat} we obtain a result of Berenstein-Zelevinsky and Stembridge:

\begin{Corollary}\cite{BZ96, Stem96} 
The action of the longest element  on the Kazhdan-Lusztig basis recovers the Sch\"utzenberger evacuation operator, i.e. for $w_0 \in S_n, \lambda \vdash n$ and $T\in \SYT(\lambda)$, we have that 
$w_0\cdot C_T =\pm C_{e(T)}$.
\end{Corollary}

Similarly we can prove a result of Rhoades regarding the action of the long cycle  $c_n=(1,2,\hdots,n)$ on the Kazhdan-Lusztig basis.  Note that in the statement below the significance of $\lambda$ being rectangular  is that the restriction of $S(\lambda)$ to $S_{n-1}$ in this case remains irreducible. 

\begin{Proposition}(cf. \cite[Proposition 3.5]{Rhoad})
Let $\lambda=(a^b)\vdash n$ be a rectangular partition.  Then for any $T \in \SYT(\lambda)$, the action of the long cycle on the Kazhdan-Lusztig basis element $C_T$ recovers the promotion operator:
$$
c_n\cdot C_T=\pm C_{j(T)}.
$$ 
\end{Proposition}

\begin{proof}
As above, set $J=\{1,\hdots,n-2\} \subset I$.  Recall that, since $\cL(\la)$ is a simple categorification, we know  $\Theta_{w_0}:D^b( \cL(\lambda)_0) \to D^b( \cL(\lambda)_0)$ is t-exact up to shift by Theorem \ref{thm:mincat}. However, $L(\la)|_{\fsl_{n-1}}$ is no longer simple, so a priori we only know that $\Theta_{w_0^J}:D^b( \cL(\lambda)_0) \to D^b( \cL(\lambda)_0)$ is perverse, but not necessarily t-exact up to shift.
We first prove that $\Theta_{w_0^J}$ is indeed t-exact up to shift.  

Consider the set of functors which are monomials in the Chevalley functors indexed by $J$:
\begin{align*}
\bM=\{\sG_{j_1}\cdots \sG_{j_s}\;|\; \sG \in \{\sE,\sF \},\;  j_\ell \in J \}.
\end{align*}
Let $\cC$ be the abelian category generated by $\{\sM(X)\;|\; X \in \cL(\lambda)_0,\sM \in \bM\}$, that is, $\cC$ is the category closed under subobjects and quotients of objects of the form $\sM(X)$.  Since the Chevalley functors are exact, it is easy to see that $\cC$ is a categorical representation of $U_q(\fsl_{n-1})$.  

Let $\nu \vdash (n-1)$ be the (unique) partition obtained from $\lambda$ by removing a box.  We claim that $\cC$ is a categorification of $L(\nu)$.  Note that this is  a categorification of the fact that $L(\nu) \cong U_q(\fsl_{n-1})\cdot L(\lambda)_0$.  Indeed, it is clear that $[\cC]_\CC$ contains $L(\nu)$.  On the other hand, if $\cC_\mu\neq0$ then $\mu$ is in the root lattice of $\fsl_{n-1}$, and $L(\nu)$ is the unique constituent of $L(\lambda)|_{\fsl_{n-1}}$ whose weights are in the root lattice.  

Now observe that $\cC_0=\cL(\lambda)_0$. Hence, by Corollary \ref{cor:mainthm} it follows that $\Theta_{w_0^J}:D^b( \cL(\lambda)_0) \to D^b( \cL(\lambda)_0)$ is t-exact up to shift.  Since $\Theta_{w_0}$ is also t-exact up to shift, it follows that $\Theta_{c_n}$ is too.  The result now follows by Proposition \ref{prop:promo} and Lemma \ref{lem:KLact}.
\end{proof}

\maketitle

\bibliography{./monbib}
\bibliographystyle{amsalpha}

\end{document}